\documentclass[10pt,oneside]{amsart}

\usepackage[
paper=a4paper,
text={138mm,208mm},centering
]{geometry}

\usepackage{amsfonts, amstext, amsmath, amsthm, amscd, amssymb}
\usepackage{graphics, pinlabel, color}
\usepackage[all,graph]{xy}
\usepackage[bookmarks]{hyperref}
\usepackage{float, placeins, booktabs, longtable, array}

\renewcommand{\setminus}{{\smallsetminus}}

\newcommand{\ZZ}{{\mathbb{Z}}}

\newcommand{\QQ}{{\mathbb{Q}}}

\DeclareMathOperator{\Id}{Id}
\DeclareMathOperator{\lk}{lk}

\DeclareMathOperator{\sign}{sign}
\DeclareMathOperator{\im}{im}

\def\co{\colon}

\def\Z{\ZZ}

\def\Q{\QQ}

\def\L{\Lambda}
\def\ol{\overline}
\def\ba{\begin{array}}
\def\ea{\end{array}}
\def\bp{\begin{pmatrix}}
\def\ep{\end{pmatrix}}

\def\sp{\operatorname{sp}}

\def\wt{\widetilde}

\def\sm{\setminus}

\def\a{\alpha}
\def\b{\beta}
\def\ll{\langle}
\def\rr{\rangle}

\def\ord{\operatorname{ord}}
\def\be{\begin{equation}} \def\ee{\end{equation}}

\theoremstyle{plain}
\newtheorem{theorem}{Theorem}[section]
\newtheorem{corollary}[theorem]{Corollary}
\newtheorem{lemma}[theorem]{Lemma}

\newtheorem{proposition}[theorem]{Proposition}
\newtheorem*{claim}{Claim}
\newtheorem*{theorem*}{Theorem}
\newtheorem*{theorem11}{Theorem~\ref{theorem:covering-link-lk-one-ribbon-intro}}

\newtheorem*{namedtheorem}{\theoremname}
\newcommand{\theoremname}{testing}

\theoremstyle{definition}

\newtheorem{definition}[theorem]{Definition}
\newtheorem{remark}[theorem]{Remark}

\iftrue
\makeatletter
\def\@maketitle{%
  \normalfont\normalsize
  \@adminfootnotes
  \@mkboth{\@nx\shortauthors}{\@nx\shorttitle}%
  \global\topskip42\p@\relax 
  \@settitle
  \ifx\@empty\authors \else \@setauthors \fi
  \ifx\@empty\@dedicatory
  \else
    \vskip.5em
    \baselineskip18\p@
    \vtop{\raggedright{\footnotesize\itshape\@dedicatory\@@par}%
      \global\dimen@i\prevdepth}\prevdepth\dimen@i
  \fi
  \@setabstract
  \normalsize
  \if@titlepage
    \newpage
  \else
    \dimen@34\p@ \advance\dimen@-\baselineskip
    \vskip\dimen@\relax
  \fi
}
\def\@settitle{%
  \vspace*{-20pt}
  \begin{flushleft}%
    \baselineskip14\p@\relax
    \normalfont\bfseries\LARGE
    \@title
  \end{flushleft}%
}
\def\@setauthors{%
  \begingroup
  \def\thanks{\protect\thanks@warning}%
  \trivlist
  \large \@topsep30\p@\relax
  \advance\@topsep by -\baselineskip
  \item\relax
  \author@andify\authors
  \def\\{\protect\linebreak}%
  \authors
  \ifx\@empty\contribs
  \else
    ,\penalty-3 \space \@setcontribs
    \@closetoccontribs
  \fi
  \normalfont\raggedright
  \@setaddresses
  \endtrivlist
  \endgroup
}
\def\@setaddresses{\par
  \nobreak \begingroup
  \small
  \def\author##1{\nobreak\addvspace\smallskipamount}%
  \def\\{\unskip, \ignorespaces}%
  \interlinepenalty\@M
  \def\address##1##2{\begingroup
    \par\addvspace\bigskipamount\noindent
    \@ifnotempty{##1}{(\ignorespaces##1\unskip) }%
    {\ignorespaces##2}\par\endgroup}%
  \def\curraddr##1##2{\begingroup
    \@ifnotempty{##2}{\nobreak\noindent\curraddrname
      \@ifnotempty{##1}{, \ignorespaces##1\unskip}\/:\space
      ##2\par}\endgroup}%
  \def\email##1##2{\begingroup
    \@ifnotempty{##2}{\nobreak\noindent E-mail address%
      \@ifnotempty{##1}{, \ignorespaces##1\unskip}\/:\space
      \ttfamily##2\par}\endgroup}%
  \def\urladdr##1##2{\begingroup
    \def~{\char`\~}%
    \@ifnotempty{##2}{\nobreak\noindent\urladdrname
      \@ifnotempty{##1}{, \ignorespaces##1\unskip}\/:\space
      \ttfamily##2\par}\endgroup}%
  \addresses
  \endgroup
  \global\let\addresses=\@empty
}
\def\@setabstracta{%
    \ifvoid\abstractbox
  \else
    \skip@25\p@ \advance\skip@-\lastskip
    \advance\skip@-\baselineskip \vskip\skip@
    \box\abstractbox
    \prevdepth\z@ 
    \vskip-10pt
  \fi
}
\renewenvironment{abstract}{%
  \ifx\maketitle\relax
    \ClassWarning{\@classname}{Abstract should precede
      \protect\maketitle\space in AMS document classes; reported}%
  \fi
  \global\setbox\abstractbox=\vtop \bgroup
    \normalfont\small
    \list{}{\labelwidth\z@
      \leftmargin0pc \rightmargin\leftmargin
      \listparindent\normalparindent \itemindent\z@
      \parsep\z@ \@plus\p@
      
    }%
    \item[\hskip\labelsep\bfseries\abstractname.]%
}{%
  \endlist\egroup
  \ifx\@setabstract\relax \@setabstracta \fi
}
\def\section{\@startsection{section}{1}%
  \z@{-1.2\linespacing\@plus-.5\linespacing}{.8\linespacing}%
  {\normalfont\bfseries\Large}}
\def\subsection{\@startsection{subsection}{2}%
  \z@{-.8\linespacing\@plus-.3\linespacing}{.3\linespacing\@plus.2\linespacing}%
  {\normalfont\bfseries}}
\def\subsubsection{\@startsection{subsection}{3}%
  \z@{.7\linespacing\@plus.2\linespacing}{-1.5ex}%
  {\normalfont\itshape}}
\def\@secnumfont{\bfseries}
\makeatother
\fi 

\def\to{\mathchoice{\longrightarrow}{\rightarrow}{\rightarrow}{\rightarrow}}
\makeatletter
\newcommand{\shortxra}[2][]{\ext@arrow 0359\rightarrowfill@{#1}{#2}}
\def\longrightarrowfill@{\arrowfill@\relbar\relbar\longrightarrow}
\newcommand{\longxra}[2][]{\ext@arrow 0359\longrightarrowfill@{#1}{#2}}
\renewcommand{\xrightarrow}[2][]{\mathchoice{\longxra[#1]{#2}}%
  {\shortxra[#1]{#2}}{\shortxra[#1]{#2}}{\shortxra[#1]{#2}}}
\makeatother

\makeatletter
\def\Nopagebreak{\@nobreaktrue\nopagebreak}
\makeatother

\begin{document}
\title{Splitting numbers of links}

\author{Jae Choon Cha}
\address{
  Department of Mathematics\\
  POSTECH\\
  Pohang 790--784\\
  Republic of Korea\\
  and\linebreak
  School of Mathematics\\
  Korea Institute for Advanced Study \\
  Seoul 130--722\\
  Republic of Korea
}
\email{jccha@postech.ac.kr}

\author{Stefan Friedl}
\address{
  Mathematisches Institut\\
  Universit\"at zu K\"oln\\
  50931 K\"oln\\
  Germany}
\email{sfriedl@gmail.com}

\author{Mark Powell}
\address{
  Department of Mathematics\\
  Indiana University \\
  Bloomington, IN 47405\\
  USA
}
\email{macp@indiana.edu}

\def\subjclassname{\textup{2010} Mathematics Subject Classification}
\expandafter\let\csname subjclassname@1991\endcsname=\subjclassname
\expandafter\let\csname subjclassname@2000\endcsname=\subjclassname


\begin{abstract}
  The splitting number of a link is the minimal number of crossing
  changes between different components required, on any diagram, to
  convert it to a split link.  We introduce new techniques to compute
  the splitting number, involving covering links and Alexander
  invariants.  As an application, we completely determine the
  splitting numbers of links with 9 or fewer crossings.  Also, with
  these techniques, we either reprove or improve upon the lower bounds
  for splitting numbers of links computed by J.~Batson and C.~Seed
  using Khovanov homology.
\end{abstract}


\maketitle

\section{Introduction}

Any link in $S^3$ can be converted to the split union
of its component knots by a sequence
of crossing changes between different
components.
Following J. Batson and C. Seed~\cite{batson-seed}, we
define the \emph{splitting number} of a link $L$, denoted by $\sp(L)$, as the
minimal number of crossing changes in such a sequence.

We present two new techniques for obtaining lower bounds for the splitting
number.  The first approach uses covering links, and the
second method arises from the multivariable Alexander polynomial of
a link.

Our general covering link theorem is stated as Theorem~\ref{theorem:general-covering-genus-lower-bound}.  Theorem~\ref{theorem:covering-link-lk-one-ribbon-intro} below gives a special case which applies to 2-component links $L$ with unknotted components and odd linking number.  Note that the splitting number is equal to the linking number modulo two.  If we take the 2-fold branched cover of $S^3$ with branching set a component of $L$, then the
preimage of the other component is a knot in $S^3$, which we
call a \emph{2-fold covering knot} of~$L$.  Also recall that the
\emph{slice genus} of a knot $K$ in $S^3$ is defined to be the minimal
genus of a surface $F$ smoothly embedded in $D^4$ such that $\partial
(D^4,F)=(S^3, K)$.

\begin{theorem}\label{theorem:covering-link-lk-one-ribbon-intro}
  Suppose $L$ is a 2-component link with unknotted components.  If
  $\sp(L)=2k+1$, then any 2-fold covering knot of $L$ has slice genus
  at most~$k$.
\end{theorem}

Theorem~\ref{theorem:general-covering-genus-lower-bound} also
has other useful consequences, given in
Corollaries~\ref{corollary:splitting-number-2-covering-link} and
\ref{corollary:covering-link-zero-lk-no}, dealing with the case of even linking numbers, for example.
Three covering link arguments which use these corollaries are
given in Section~\ref{section:arguments-splitting-no-particular-links}.

Our Alexander polynomial method is efficacious for 2-component links
when the linking number is one and at least one
component is knotted.  By looking at the effect of a crossing change
on the Alexander module we obtain the following result:

\begin{theorem}\label{prop:split1-intro}
  Suppose $L$ is 2-component link with Alexander
  polynomial~$\Delta_L(s,t)$.  If $\sp(L)=1$, then $\Delta_L(s,1)\cdot
  \Delta_L(1,t)$ divides $\Delta_L(s,t)$.
\end{theorem}

We will use elementary methods explained in Lemma~\ref{lem:splittingnumberbasics} and our techniques from covering links and Alexander polynomials to obtain lower bounds on the splitting number for links with 9 or fewer crossings.
Together with enough patience with
link diagrams, this is sufficient to determine the
splitting number for all of these links.  Our results for links up
to 9 crossings are summarised by Table~\ref{table:splitting-numbers}
in Section~\ref{section:table-splitting-no-results}.

In~\cite{batson-seed}, Batson and Seed defined a spectral sequence from
the Khovanov homology of a link $L$ which converges to the Khovanov
homology of the split link with the same components as~$L$. They
showed that this spectral sequence gives rise to a lower bound on
$\sp(L)$, and by computing it for links up to 12 crossings, they gave
many examples for which this lower bound is strictly stronger than the
lower bound coming from linking numbers.  They determined the
splitting number of some of these examples, while some were left
undetermined.

We revisit the examples of Batson and Seed and show that our methods
are strong enough to recover their lower bounds.  Furthermore we show
that for several cases our methods give more information.  In
particular, we completely determine the splitting numbers of all the
examples of Batson and Seed.  We refer the reader to
Section~\ref{section:batson-seed-examples} for more details.

\subsection*{Organisation of the paper}

We start out, in Section \ref{section:basics}, with some basic
observations on the splitting number of a link.
In Section~\ref{section:covering-link-technique} we prove Theorem~\ref{theorem:general-covering-genus-lower-bound}, which is a general result on the effect of crossing changes on covering links,
 and then we provide an
example in Section~\ref{section:exampleII}.  We give a proof of
Theorem~\ref{prop:split1-intro} in
Sections~\ref{section:effect-on-alexander-modules} and~\ref{section:alex-poly-obstructions-to-splitting} and we
illustrate its use with an example in
Section~\ref{section:splitting-no-examples}.  The examples of Batson
and Seed are discussed in Section~\ref{section:batson-seed-examples},
with Section~\ref{section:alex-poly-batson-seed} focussing on examples
which use Theorem~\ref{prop:split1-intro}, and
Section~\ref{section:batson-seed-covering-link-technique} on examples
which require Theorem~\ref{theorem:covering-link-lk-one-ribbon-intro}.
A 3-component example of Batson and Seed is discussed in
Section~\ref{section:batson-seed-3-comp-example}.  Next, our results
on the splitting numbers of links with 9 crossings or fewer are given
in Section~\ref{section:table-splitting-no-results}, with some particular
arguments used to obtain these results described in
Section~\ref{section:arguments-splitting-no-particular-links}.

\subsection*{Acknowledgements}

Part of this work was completed while the authors were on visits, JCC and SF
at Indiana University in Bloomington and MP at the Max Planck
Institute for Mathematics in Bonn.  The authors thank the
corresponding institutions for their hospitality.  We are particularly
grateful to Daniel Ruberman whose suggestion to work on splitting numbers was invaluable.
We would also like to thank Maciej Borodzik, Matthias Nagel, Kent Orr and Raphael
Zentner for their interest and suggestions.  Julia Collins and Charles
Livingston provided us with a Maple program to compute twisted
Alexander polynomials (also written by Chris Herald and Paul Kirk),
and helped us to use it correctly.

JCC was partially supported by NRF grants 2010--0029638 and
2012--0009179.  MP gratefully
acknowledges an AMS-Simons travel grant which helped with his travel
to Bonn.

\section{Basic observations}\label{section:basics}

A link is \emph{split} if it is a split union of knots.  We recall from the introduction that the
\emph{splitting number $\sp(L)$} of a link $L$ is defined to be the
minimal number of crossing changes which one needs to make on $L$,
each crossing change between different components, in order to obtain
a split link.

We note that this differs from the definition of `splitting number' which occurs in
\cite{Adams96, Shimizu12}; in these papers crossing changes of a
component with itself are permitted.

Given a link $L$ we say that a non-split sublink with all of the linking numbers zero is \emph{obstructive}.  (All obstructive sublinks which occur in the applications of this paper will be Whitehead links.)  We then define $c(L)$
to be the maximal size of a collection of distinct obstructive sublinks of $L$, such that any two sublinks in the collection have at most one component in common.  Note that~$c$ is zero for trivial links.

 As another example consider the link $L9a54$ shown in Figure~\ref{figure:L9a54}.  The sublink $L_1 \sqcup L_3$ is an unlink, while both $L_1 \sqcup L_2$ and $L_2 \sqcup L_3$ are Whitehead links, hence are obstructive.  Thus $c(L)=2$.

\begin{figure}[H]
      \labellist
      \small\hair 0mm
      \pinlabel {$L_1$} at 15 90
      \pinlabel {$L_2$} at 180 150
      \pinlabel {$L_3$} at 345 90
      \endlabellist
      \begin{center}
      \includegraphics[scale=.55]{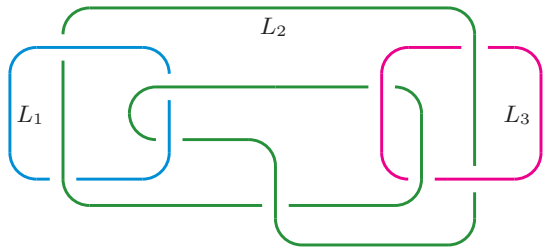}
      \end{center}
      \caption{The link $L9a54$.}
      \label{figure:L9a54}
\end{figure}

Finally we discuss the link $J$ in Figure~\ref{figure:L9a46-covering-link-over-L3}. It
has four components $J_1,J_2,J_3,J_4$ and $J_1\cup J_3$ and $J_2\cup
J_4$ each form a Whitehead link. It follows that $c(J)=2$.

In practice it is straightforward to obtain lower bounds for $c(L)$.  In most cases it is also not too hard to determine $c(L)$ precisely.

Now we have the following elementary lemma.

\begin{lemma}\label{lem:splittingnumberbasics}
Let $L=L_1\,\sqcup \dots\sqcup\, L_m$ be a link. Then
\[ \sp(L)\equiv \sum_{i>j} \lk(L_i,L_j) \mod 2\]
and
\[ \sp(L)\geq \sum_{i>j} |\lk(L_i,L_j)|+2c(L).\]
\end{lemma}

\begin{proof}
Given a link $L$ we write
\[ a(L)=\sum_{i > j}|\lk(L_i,L_j)|.\]
Note that a crossing change between two different components always changes the value of $a$ by precisely one. Since $a$ of the unlink is zero we immediately obtain the first statement.

If we do a crossing change between two components with non-zero
linking number, then $a$ goes down by at most one, whereas $c$ stays
the same or increases by one.
On the other hand, if we do a crossing change between two components with zero linking number, then $a$ goes up by one and $c$ decreases by at most one, since the two components belong to at most one obstructive sublink in any maximal collection whose cardinality realises $c(L)$. It now follows that $a(L)+2c(L)$ decreases with each crossing change between different components by at most one.
\end{proof}

The right hand side of the second inequality is greater than or equal to the lower bound $b_{\lk}(L)$ of~\cite[Section~5]{batson-seed}.  In some cases the lower bound coming from Lemma~\ref{lem:splittingnumberbasics} is stronger.  For example, let $L$ be two split copies of the Borromean rings.  For this $L$ we have $c(L)=2$, giving a sharp lower bound on the splitting number of $4$, whereas $b_{\lk}(L)=2$.

\section{Covering link calculus}

In this section, first we prove our main covering link result, Theorem~\ref{theorem:general-covering-genus-lower-bound}, showing that covering links can be used to give lower bounds on the splitting number.  Then we show how to extract Theorem~\ref{theorem:covering-link-lk-one-ribbon-intro} and three other useful corollaries from Theorem~\ref{theorem:general-covering-genus-lower-bound}.  In Section~\ref{section:exampleII} we present an example of this approach.

\subsection{Crossing changes and covering links}\label{section:covering-link-technique}

The following definition is a special case of the notion of a covering
link occurring in~\cite[Method~5]{Kohn93} and \cite{Cha-Kim:2008-1}, for example.

\begin{definition}\label{defn:covering-link-and-height}
  Let $L=L_1\,\sqcup \dots \sqcup \, L_m$ be an $m$-component link
  with $L_i$ unknotted.  We denote the double branched cover of $S^3$
  with branching set the unknot $L_i$ by $p \colon S^3\to S^3$.  We
  refer to $p^{-1}(L\sm L_i)$ as the \emph{2-fold covering link of $L$
    with respect to~$L_i$}.
\end{definition}

We note that a choice of orientation of a link induces an orientation
of its covering links.


In the theorem below we use the term \emph{internal band sum} to refer to the operation on an oriented link $L$ described as follows.  The data for the move is an embedding $f \colon D^1 \times D^1 \subset S^3$ such that $f(D^1 \times D^1) \cap L = f(\{-1,1\} \times D^1)$, the orientation of $f(\{-1\} \times D^1)$ agrees with that of $L$ and the orientation of $f(\{1\} \times D^1)$ is opposite to that of $L$.  The output is a new oriented link given by $(L \setminus f(\{-1,1\} \times D^1)) \cup f(D^1 \times \{-1,1\})$, after rounding corners.  The new link has the orientation induced from~$L$.

\begin{theorem}
  \label{theorem:general-covering-genus-lower-bound}
  Let $L=L_1\sqcup\cdots\sqcup L_m$ be an $m$-component link and suppose that $L_i$ is unknotted for some fixed~$i$.  Fix
  an orientation of~$L$.  Suppose $L$ can be transformed to a split link
  by $\a+\b$ crossing changes involving distinct components, where $\a$
  of these involve $L_i$ and $\b$ of these do not involve~$L_i$.  Then
  the 2-fold covering link $J$ of $L$ with respect to $L_i$ can be altered by performing $\a$ internal band sums and $2\b$~crossing changes between different components to the split union of $2(m-1)$ knots comprising two copies of $L_j$, for each $j \neq i$.
 \end{theorem}

  \begin{proof}
  We may assume $i=1$.  We begin by investigating the effect of
  crossing changes on the 2-fold covering link with respect to the
  first component $L_1$ of a link~$L$.

  \begin{trivlist}
  \item[] \emph{Type A.}  First we consider crossing changes between
    the branching component $L_1$ and another component, say~$L_2$.
    Such a crossing change lifts to a rotation of the preimage $J$ of
    $L_2$ around the lift $\wt{L_1}$ of $L_1$, as shown in
    Figure~\ref{figure:2-fold-cover-branching-set-ribbon}.  The top
    left and middle left diagrams show a link before and after a
    crossing change, in a cylindrical neighbourhood which contains an
    interval from each of $L_1$ and $L_2$.  To branch over $L_1$,
    which is the component running down the centre of the cylinders,
    cut along the surface which is shown in the diagrams.  The results
    of taking the top left and middle left diagrams, cutting, and
    glueing two copies together, are shown in the top right and middle
    right diagrams respectively.

    \begin{figure}[H]
      \labellist
      \small\hair 0mm
      \pinlabel {$L_2$} at 23 568
      \pinlabel {$L_1$} at 68 568
      \pinlabel {$J$} at 173 568
      \pinlabel {$\wt{L_1}$} at 218 568
      \pinlabel {$J$} at 258 568
      \endlabellist
      \begin{center}
        \includegraphics[scale=.55]{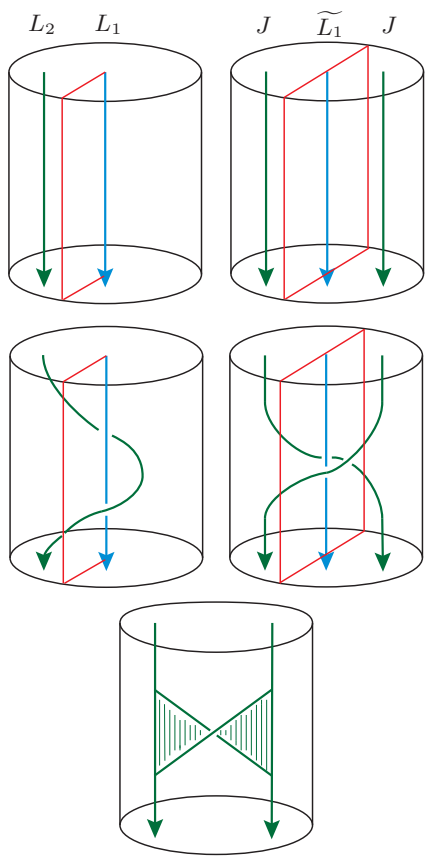}
      \end{center}
      \caption{The effect of a crossing change on a 2-fold covering link where one component is the branching set.}
      \label{figure:2-fold-cover-branching-set-ribbon}
    \end{figure}

    After forgetting the branching set, the same effect on the lift of
    $L_2$ can be achieved by adding a band to $J$; see the bottom
    diagram of Figure~\ref{figure:2-fold-cover-branching-set-ribbon}.
    By ignoring the band, we obtain the top right diagram with the
    branching component removed.  If we instead use the band to make an internal band sum, we obtain the middle right diagram with the branching
    set removed.  Note that this band is attached to $J$ in such a way
    that orientations are preserved.  This holds no matter
    what choice of orientations were made for~$L$.  Thus we see that a crossing change between
    $L_1$ and $L_2$ corresponds to an internal band sum
    on the covering link.
  \end{trivlist}

  \begin{trivlist}
  \item[] \emph{Type B.}  Consider a crossing change which does not
    involve $L_1$, say between $L_2$ and~$L_3$.  Such a crossing
    change can be realised by $\pm 1$ Dehn surgery on a circle which
    has zero linking number with $L$ and which bounds an embedded
    disc, say $D$, in $S^3$ that intersects $L$ in two points of opposite signs, one point of $L_2$
    and one point of~$L_3$.  By performing the Dehn surgeries, and
    then taking the branched cover over $L_1$, we produce the covering
    link of the link obtained by the crossing change.

    Note that the preimage of the disc $D$ in the double branched
    cover consists of 2 disjoint discs, each of which intersects the covering link
    transversally in two points with opposite signs, one point of the preimage of~$L_2$
    and one point of the preimage of~$L_3$.  As an
    alternative construction, we can take the branched cover and then
    perform $\pm1$ Dehn surgeries along the boundary circles of the
    preimage discs.  This gives the same the covering link.  From
    this it follows that a single crossing change between $L_2$ and
    $L_3$ corresponds to two crossing changes on the covering link.
  \end{trivlist}

  Note that when there is more than one crossing change, of either type, the corresponding surgery discs and bands associated to the covering link are disjoint.

  Recall that the link $L$ can be altered to become the split union of
  $m$ knots $L_1,\ldots,L_m$ by $\a$ crossing changes of Type~A and $\b$
  crossing changes of Type~B\hbox{}.  By the above arguments, the
  2-fold covering link of $L$ with respect to the first component $L_1$
  can be altered to become the corresponding covering link of the split link,
  which is the split union $L_2 \,\sqcup\, L_2\, \sqcup \cdots \sqcup\, L_m
  \,\sqcup\, L_m$, by $\a$ internal band sums and $2\b$ crossing changes.  
\end{proof}

In the following result, $g_4(K)$ denotes the slice genus of a knot
$K$ in~$S^3$, namely the minimal genus of a smoothly embedded
connected oriented surface in $D^4$ whose boundary is~$K$.

  \begin{corollary}\label{cor:covering-genus-lower-bound}
  Under the same hypotheses as Theorem~\ref{theorem:general-covering-genus-lower-bound},
   the 2-fold covering link of $L$ with respect to $L_i$ bounds a
  smoothly embedded oriented surface $F$ in $D^4$ which has no closed
  components and has Euler characteristic
  \[
  \chi(F)=2(m-1)-\a -4\b - 4\sum_{k\ne i} g_4(L_k).
  \]
  In addition, if there is some $j\ne i$ such that each $L_k$ with
  $k\ne j$ is involved in some crossing change with $L_j$, then $F$ is
  connected.
\end{corollary}

\begin{proof}
Once again we may assume that $i=1$.  Let $J$ be the 2-fold covering link of $L$ with respect to $L_1$.

An internal band sum can be inverted by performing another band sum, while the inverse of a crossing change is also a crossing change.  Hence by Theorem~\ref{theorem:general-covering-genus-lower-bound} we can also obtain the covering link~$J$ from the split union $L_2 \,\sqcup\, L_2\, \sqcup \cdots \sqcup\, L_m  \,\sqcup\, L_m$ by performing $\a$ internal band sums and $2\b$ crossing changes.

Choose surfaces $V_j$ embedded in $D^4$ with $\partial V_j = L_j$ and genus
  $g_4(L_j)$.  Take a split union $V_2\, \sqcup \, V_2\, \sqcup \cdots
  \sqcup\, V_m\, \sqcup\, V_m$ in $D^4$.  The boundary of these surfaces is the split union $L_2 \,\sqcup\, L_2\, \sqcup \cdots \sqcup\, L_m
  \,\sqcup\, L_m$.  The covering link $J$ can be realised as the boundary of a surface obtained from the split union of the surfaces by attaching $\a$ bands and $2\b$ clasps in $S^3$.  As pointed out in the proof of Theorem~\ref{theorem:general-covering-genus-lower-bound}, the surgery discs and bands associated to crossing changes are disjoint.  Pushing slightly into $D^4$,
  we obtain an immersed surface in $D^4$ bounded by~$J$; each clasp
  gives a transverse intersection.  As usual, we remove the
  intersections by cutting out a disc neighbourhood of the intersection
  point from each sheet and glueing a twisted annulus which is a
  Seifert surface for the Hopf link.  This gives a smoothly embedded
  oriented surface $F$ in $D^4$ bounded by the covering link~$J$.
  Note that each band attached changes the Euler characteristic
  of the surface by $-1$, while each twisted annulus used to remove an
  intersection point changes the Euler characteristic by~$-2$.
  Therefore the resulting surface $F$ has Euler characteristic
  \[
  \chi(F) = \sum_{k=2}^m 2\cdot (1-2g_4(V_k)) - \a - 4\b
  \]
  which is equal to the claimed value.

The final conclusion of the corollary states (when $i=1$) that $F$ is connected if there is some $j\ne 1$ such that each $L_k$ with
  $k\ne j$ is involved in some crossing change with $L_j$.  To see this, observe that a crossing change involving
  $L_j$ and $L_1$ joins the two copies of $V_j$; a crossing change
  involving $L_j$ and $L_k$ with $j$, $k \geq 2$ joins one of the
  two copies of $V_j$ to one of the two copies of $V_k$ and joins
  the other copy of $V_j$ to the other copy of~$V_k$.  Under the
  hypothesis, it follows that $F$ is connected.
\end{proof}

Corollary~\ref{cor:covering-genus-lower-bound} has some useful
consequences of its own.  Considering the case
of $m=2$, $\a=2k+1$, $\b=0$, and $g_4(L_k)=0$, we obtain
Theorem~\ref{theorem:covering-link-lk-one-ribbon-intro} stated in the
introduction.

\begin{theorem11}
Suppose $L$ is a 2-component link with unknotted
  components.  If $\sp(L)=2k+1$, then any 2-fold covering knot of $L$
  has slice genus at most~$k$.
\end{theorem11}

\begin{remark}\leavevmode\Nopagebreak
\begin{enumerate}
\item In the proof of
Corollary~\ref{cor:covering-genus-lower-bound}, when $\b=0$,
we construct an embedded surface $F$ without local maxima.
Therefore in order to show, using
Theorem~\ref{theorem:covering-link-lk-one-ribbon-intro}, that a link of linking number one with
unknotted components has splitting number at least three, it suffices
 to show that the covering link is not a ribbon knot.
\item Different choices of orientation on a link $J$ can change the minimal genus of a connected surface which $J$ bounds in $D^4$.  Since the splitting number is independent of orientations, in applications we will choose the orientation which gives the strongest lower bound.  This remark will be relevant in Section~\ref{section:batson-seed-3-comp-example}.
\item  If $L$ is a non-split 2-component link, then the surface $F$ of Corollary~\ref{cor:covering-genus-lower-bound}
is automatically connected, by the last sentence of that corollary.
\end{enumerate}
\end{remark}

The following is another useful consequence of Corollary~\ref{cor:covering-genus-lower-bound}.

\begin{corollary}
  \label{corollary:splitting-number-2-covering-link}
  Suppose $L$ is a 2-component link with unknotted components and
  $\sp(L)=2$.  Then any 2-fold covering link of $L$ is weakly slice;
  that is, bounds an annulus smoothly embedded in~$D^4$.
\end{corollary}

\begin{proof}
  First note that a 2-fold covering link has two components, since the
  linking number is even by Lemma~\ref{lem:splittingnumberbasics}.
  Applying Corollary~\ref{cor:covering-genus-lower-bound} with $m=2$,
  $\a=2$, $\b=0$, and $g_4(L_k)=0$, the conclusion follows.
\end{proof}

We state one more corollary to
Theorem~\ref{theorem:general-covering-genus-lower-bound}.  Let
$\sp_i(L)$ be the minimal number of crossing changes between distinct
components not involving $L_i$ required to transforms $L$ to a split
link.  By convention, $\sp_i(L)$ is infinite if we must make a crossing change involving $L_i$ in order to split~$L$.

\begin{corollary}[c.f.\ {\cite[Method~5]{Kohn93}}]
  \label{corollary:covering-link-zero-lk-no}
  For a link $L=L_1\sqcup\cdots\sqcup L_m$ and its 2-fold covering
  link $J$ with respect to $L_i$, we have $\sp_i(L) \ge \frac12
  \sp(J)$.
\end{corollary}

\begin{proof}
This follows from Theorem~\ref{theorem:general-covering-genus-lower-bound} with $\a=0$.
\end{proof}

We remark that the above results generalise to $n$-fold covering links in a reasonably straightforward manner.  One can also draw analogous conclusions when the branching component is knotted.  We do not address these generalisations here, since the results stated above are sufficient for the applications considered in this paper.

\subsection{An example of the covering link technique}\label{section:exampleII}

To illustrate the use of the method developed in
Section~\ref{section:covering-link-technique}, we now apply it to
prove that the splitting number of the 2-component link $L9a30$
is three.  More applications of Theorem~\ref{theorem:covering-link-lk-one-ribbon-intro} and Corollaries~\ref{corollary:splitting-number-2-covering-link} and \ref{corollary:covering-link-zero-lk-no} will be discussed later; see
Sections~\ref{section:batson-seed-covering-link-technique},~\ref{section:batson-seed-3-comp-example},~\ref{section:table-splitting-no-results},
and~\ref{section:arguments-splitting-no-particular-links} for
instance.  In this paper, we use the link names employed in the
LinkInfo database~\cite{linkinfo}.  The link $L9a30$ is shown in
Figure~\ref{figure:L9a30}.  It is a two component link of linking
number one with unknotted components.  Recall that the splitting
number is determined modulo 2 by the linking number by
Lemma~\ref{lem:splittingnumberbasics}.  It is easy to see from
Figure~\ref{figure:L9a30} that $3$ crossing changes suffice, so the
splitting number is either one or three.

\begin{figure}[H]
\begin{center}
\includegraphics[scale=.6]{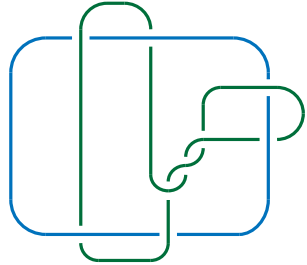}
\end{center}
\caption{The link $L9a30$.}
\label{figure:L9a30}
\end{figure}

To see that $\sp(L9a30) \neq 1$, we take a 2-fold cover branched over
one of the components, and check that the resulting knot is not slice.
Figure~\ref{figure:L9a30-covering-knot} shows the result of an isotopy
which was made in preparation for taking a branched cover on the left,
and the knot obtained as the preimage of $L9a30$ after deleting the
preimage of the branching component on the right.

\begin{figure}[H]
\begin{center}
\includegraphics[scale=0.6]{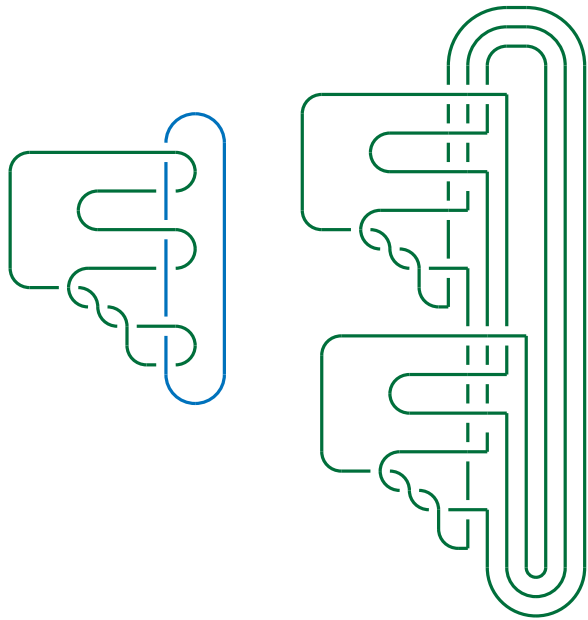}
\end{center}
\caption{Left: the link $L9a30$ after an isotopy to prepare for taking
  the cover branching over the most obviously unknotted component.
  Right: the knot which arises as the covering link after taking
  two-fold branched cover and deleting the branching set.}
\label{figure:L9a30-covering-knot}
\end{figure}

The knot on the right of Figure~\ref{figure:L9a30-covering-knot} after
a simplifying isotopy is shown in
Figure~\ref{figure:L9a30-covering-knot-isotopied}; it is a twist knot
with a negative clasp and 7 positive half twists.  This knot is well
known not to be a slice knot, a fact which was first proved by
A.~Casson and C.~Gordon~\cite{CassonGordon2, CassonGordon}.
Therefore, by Theorem~\ref{theorem:covering-link-lk-one-ribbon-intro},
the splitting number of $L9a30$ is at least three, as claimed.

\begin{figure}[H]
\begin{center}
\includegraphics[scale=.5]{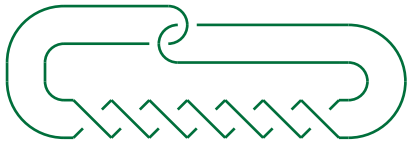}
\end{center}
\caption{The covering knot on the right of Figure~\ref{figure:L9a30-covering-knot} after an isotopy.}
\label{figure:L9a30-covering-knot-isotopied}
\end{figure}

\section{Alexander invariants}

In this section we will recall the definition of Alexander
modules and polynomials of oriented links.  We then show how Alexander modules
are affected by a crossing change which then allows us to prove
Theorem \ref{prop:split1-intro}.

\subsection{Crossing changes and the Alexander module}\label{section:effect-on-alexander-modules}

Throughout this section, given an oriented $m$-component link $L$, the oriented meridians are denoted by $\mu_1,\dots,\mu_m$.
Note that $\mu_1,\dots,\mu_m$ give rise to a basis for $H_1(S^3\setminus \nu L;\Z)$.
We will henceforth use this basis to identify $H_1(S^3\setminus \nu L;\Z)$ with $\Z^m$.
Let $R$ be a subring of $\Bbb{C}$ and let $\psi\colon \Z^m\to F$ be a homomorphism to a free abelian group.
We denote the induced map
\[ \pi_1(S^3\setminus \nu L)\to H_1(S^3\setminus \nu L;\Z)=\Z^m\xrightarrow{\psi} F\]
by $\psi$ as well. We can then consider the corresponding Alexander module
\[ H_1^\psi(S^3\setminus \nu L;R[F])\]
and the order of the Alexander module is denoted by
\[ \Delta_L^\psi\in R[F]=\ord_{R[F]}(H_1^\psi(S^3\setminus \nu L;R[F])).\]
(We refer to \cite{Hi02} for the definition of the order of a $R[F]$-module.)
If $\psi$ is the identity, then we drop $\psi$ from the notation and we obtain the usual multivariable Alexander polynomial $\Delta_L$.

Note that what we term the Alexander module has also been called the ``link module'' in the literature e.g.\ \cite{KawauchiBook96}.
The following proposition relates the Alexander modules of two oriented links which differ by a crossing change.

\begin{proposition}\label{prop:alexmodulecrossingchange}
Let $L$ and $L'$ be two oriented $m$-component links which differ by a single crossing change.
Let $R$ be a subring  of $\Bbb{C}$ and let $\psi\colon \Z^m\to F$ be a homomorphism to a free abelian group.
Then there exists a diagram
\[ \xymatrix @R-0.5cm @C-0.5cm{
&R[F]\ar[dr]&&R[F]\ar[dl]& \\
&& M \ar[dl]\ar[dr] && \\
&H_1^\psi(S^3\setminus \nu L;R[F])\ar[dl]&&H_1(S^3\setminus \nu L';R[F])\ar[dr]& \\
 0&&&&0}\]
where $M$ is some $R[F]$-module and where the diagonal sequences are exact.
\end{proposition}

The formulation of this proposition is somewhat more general than what is strictly needed in the proof of Theorem \ref{prop:split1-intro}.
We hope that this more general formulation will be applicable, in future work, to the computation of unlinking numbers; see the beginning of Section~\ref{section:table-splitting-no-results} for the definition of the unlinking number of a link.

\begin{proof}
We write $X=S^3\sm \nu L$ and $X'=S^3\sm \nu L'$.
We pick two open disjoint discs $D_1$ and $D_2$ in the interior of $D^2$ and we write
\begin{align*}
B&=(D^2\sm (D_1\cup D_2))\times [0,1],\\
S&=(D^2\sm (D_1\cup D_2))\times \{0,1\}\cup_{\partial D^2 \times\{0,1\}} S^1\times [0,1].
\end{align*}
Put differently, $S$ is the `top and bottom boundary' of $B$ together with the outer cylinder $S^1\times [0,1]$.

Since $L$ and $L'$ are related by a single crossing  change  there exists a subset  $Y$ of $X$ and  continuous injective maps $f\colon B\to X$ and $f'\colon B'\to X'$ with the following properties:

\begin{enumerate}
\item $X=Y\cup f(B)$ and $Y\cap f(B)=f(S)$,
\item $X'=Y\cup f'(B)$ and $Y\cap f'(B)=f'(S)$.
\end{enumerate}

We can now state the following claim.

\begin{claim}
There exists a short exact sequence
\[ R[F]\to  H_1^\psi(Y;R[F])\to H_1^\psi(X;R[F])\to 0.\]
\end{claim}

By a slight abuse of notation we now write $B=f(B)$ and $S=f(S)$.
We then consider the following  Mayer-Vietoris sequence:
\begin{multline*}
\cdots\to H_1^\psi(S;R[F])\xrightarrow{i_*\oplus j_*} H_1^\psi(B;R[F])\oplus
H_1^\psi(Y;R[F]) \to H_1^\psi(X;R[F])
\\
\to H_0^\psi(S;R[F])\xrightarrow{i_*\oplus j_*} H_0^\psi(B;R[F])\oplus
H_0^\psi(Y;R[F])
\end{multline*}
where $i\colon S\to B$ and $j\colon S\to Y$ are the two inclusion maps.
We need to study the relationships between the homology groups of $S$ and $B$. We make the following observations:
By \cite[Section~VI.3]{HS97} we have the following commutative diagram
\[ \xymatrix{ H_0^\psi(S;R[F])\ar[d]_{\cong} \ar[r] & H_0^\psi(B;R[F])\ar[d]_{\cong} \\
R[F]/\{ \psi(g)v-v\}_{v\in R[F],\, g\in \pi_1(S)}\ar[r] & R[F]/\{ \psi(g)v-v\}_{v\in R[F],\, g\in \pi_1(B)}.}\]
Here the horizontal maps are induced by the inclusion $S\to B$ and the vertical maps are isomorphisms.
The map $i_*\co \pi_1(S)\to \pi_1(B)$ is surjective; it follows that the bottom horizontal map is an isomorphism.
Hence the top horizontal map is also an isomorphism.
The above Mayer-Vietoris sequence thus simplifies to the following sequence
\[ H_1^\psi(S;R[F])\xrightarrow{i_*\oplus j_*} H_1^\psi(B;R[F])\oplus H_1^\psi(Y;R[F])\to H_1^\psi(X;R[F])\to 0.\]

We note that the space $B$ is homotopy equivalent to a wedge of two circles $m$ and $n$.
Furthermore,  $S$ is homotopy equivalent to the wedge of $m,n$ and another curve $m'$ which is homotopic to $m$ in $B$.
By another slight abuse of notation we now replace $B$ and $S$ by these wedges of circles and we view $B$ and $S$ as CW-complexes with precisely one 0-cell
in the obvious way. We denote by $p\co \tilde{S}\to S$ and $p\co \tilde{B}\to B$
the coverings corresponding to the homomorphisms $\pi_1(S)\to \pi_1(B)\to \pi_1(X)\xrightarrow{\psi}F$. Note that we can and will view $\tilde{B}$ as a subset of $\tilde{S}$.
We now pick pre-images $\tilde{m}, \tilde{n}$ and $\tilde{m}'$ of $m,n$ and $m'$ under the covering map  $p\colon \tilde{S}\to S$.
Note that $\{\tilde{m},\tilde{n}\}$ is a basis for $C_1(B;R[F])=C_1(\tilde{B})$ and
$\{\tilde{m},\tilde{n},\tilde{m}'\}$ is a basis for $C_1(S;R[F])=C_1(\tilde{S})$.
 The kernel of the map $C_1(S;R[F])\to C_1(B;R[F])$ is given by
$R[F]\cdot (m-m')$. We thus obtain the following commutative diagram of chain complexes with exact rows
\[ \xymatrix{ 0\ar[r]& R[F]\cdot (m-m')\ar[r]\ar[d] &C_1(S;R[F])\ar[d]\ar[r]^{i_*}& C_1(B;R[F])\ar[r]\ar[d]&0 \\
&0\ar[r]& C_0(S;R[F])\ar[r]^{i_*}&C_0(B;R[F])\ar[r]&0.}\]
It now follows easily from the diagram, or more formally from the snake lemma, that
\be \label{equ:ker} \ker(i_*\co H_1^\psi(S;R[F])\to H_1^\psi(B;R[F]))\cong R[F]\cdot(m-m')\ee
and that
\be \label{equ:coker} \operatorname{coker}(i_*\co H_1^\psi(S;R[F])\to H_1^\psi(B;R[F]))=0.\ee

Finally we consider the following commutative diagram
\[ \xymatrix @R+0.5cm { R[F]\cdot (m-m')\ar[r] \ar[d]& H_1^\psi(Y;R[F])\ar[r]\ar[d]_{\begin{pmatrix}0 \\ \Id \end{pmatrix}} & H_1^\psi(X;R[F])\ar[d]^=\ar[r]& 0 \\
 H_1^\psi(S;R[F])\ar[r] & H_1^\psi(B;R[F])\oplus H_1^\psi(Y;R[F]) \ar[r] & H_1^\psi(X;R[F])\ar[r]&0.}\]

We had already seen above that the bottom horizontal sequence is exact. It now follows from (\ref{equ:ker}), (\ref{equ:coker})
and some modest diagram chasing that the top horizontal sequence is also exact.
This concludes the proof of the claim.

Precisely the same proof shows that there exists a short exact sequence
\[ R[F]\to  H_1^\psi(Y;R[F])\to H_1^\psi(X';R[F])\to 0.\]
(Use $B=f'(B), S=f'(S)$ instead of $B=f(B), S=f(S)$). Combining these two short exact sequences now gives the desired result, by taking $M:=H_1^\psi(Y;R[F])$.
\end{proof}


\subsection{The Alexander polynomial obstruction}\label{section:alex-poly-obstructions-to-splitting}

Using Proposition \ref{prop:alexmodulecrossingchange} we can prove the following obstruction to the splitting number being equal to $1$.

\begin{theorem}\label{prop:split1}
Let $L$ be a 2-component oriented link.  We denote the Alexander polynomial of $L$ by $\Delta_L(s,t)$.
If the splitting number of $L$ equals one, then $\Delta_L(s,1)\cdot \Delta_L(1,t)$ divides $\Delta_L(s,t)$.
\end{theorem}

Let $L=J\cup K$ be an oriented link with splitting number equal to one. We denote the Alexander polynomials of $J$ and $K$ by $\Delta_J$ and $\Delta_K$ respectively.  It follows from Lemma~\ref{lem:splittingnumberbasics} that the linking number satisfies $|\lk(J,K)|=1$.
Therefore by the Torres condition $|\Delta_L(1,1)|=1$ and we have that
\[\Delta_L(s,1)=\Delta_J(s)\mbox{ and } \Delta_L(1,t)=\Delta_K(t).\]
We can thus reformulate the statement of the theorem as saying that if $L=J\cup K$ is an oriented link with splitting number equal to one,
then $\Delta_J(s)$ and $\Delta_K(t)$ both divide $\Delta_L(s,t)$.

\begin{proof}
Let $L=J\cup K$ be an oriented link with splitting number equal to one.
We denote by $\psi \colon H_1(S^3\sm L;\Z)\to \ll s,t\,|\,[s,t]=1\rr$ the map which is given by sending the meridian of $J$ to $s$ and by sending the meridian of $K$ to $t$.
We write $\L := \Z[s^{\pm 1},t^{\pm 1}]$.

In the following we also denote by $\psi$ the map $H_1(S^3\sm J;\Z)\to \ll s,t\,|\,[s,t]=1\rr$, which is given by sending the meridian of $J$ to $s$.
Note that with this convention we have an isomorphism
\[ H_1^\psi(S^3\sm J;\L)=H_1(S^3\sm J;\Z[s^{\pm 1}])\otimes_{\Z[s^{\pm 1}]}\L\]
and we obtain that
\be \label{equ:deltaj} \ord_{\L}(H_1^\psi(S^3\sm J;\L))=\Delta_J(s).\ee
Similarly we define a map $H_1(S^3\sm K;\Z)\to \ll s,t\,|\,[s,t]=1\rr$ by sending the meridian of $K$ to $t$. We see that
\be \label{equ:deltak} \ord_{\L}(H_1^\psi(S^3\sm K;\L))=\Delta_K(t).\ee

We denote the split link with components $J$ and $K$ by $J \sqcup K$.
The Mayer-Vietoris sequence for $S^3\sm (J  \sqcup  K)$ which comes from  splitting along the separating 2-sphere $S$ gives rise to an exact sequence
\begin{multline*}
0\to  H_1^\psi(S^3\sm J;\L)\oplus H_1^\psi(S^3\sm  K;\L)\to
H_1^\psi(S^3\sm ( J\sqcup K);\L)
\\
\xrightarrow{h} H_0(S;\L)\to H_0^\psi(S^3\sm J;\L)\oplus H_0^\psi(S^3\sm  K;\L).
\end{multline*}
We recall that by \cite[Section~VI]{HS97} we have, for any connected space $X$ with a ring homomorphism $\psi\co \pi_1(X)\to \L$ we have
\[ H_0^\psi(X;\L)=\L/\{\psi(g)v-v\,|\, v\in \L,g\in \pi_1(X)\}.\]
It follows easily that  $ H_0^\psi(S;\L)\cong \L$ and that  $H_0^\psi(S^3\sm J;\L)$ and $H_0^\psi(S^3\sm  K;\L)$ are $\L$-torsion.
In particular we see that the last map in the above long exact sequence  has a nontrivial kernel.
By the exactness of the Mayer-Vietoris sequence above it follows that the map $h$ has nontrivial image.

Since $L$ has splitting number one we can do one crossing change involving both $J$ and $K$ to turn $L$ into $J\sqcup K$.
The conclusion of Proposition \ref{prop:alexmodulecrossingchange} together with the above Mayer-Vietoris sequence gives rise to a  diagram of maps as follows:
\[\xymatrix{\L\ar[r]^-{f} &M\ar@{->>}[r]^-{p}\ar@{->>}[d]^-{g} & H_1^\psi(S^3\sm L;\L) \\
 H_1^\psi(S^3\sm  J;\L)\oplus H_1^\psi(S^3\sm K;\L)\ar@{^{(}->}[r]& H_1^\psi(S^3\sm ( J\sqcup  K);\L)\ar[r]^-h& \L,}\]
where the top and bottom horizontal sequences are exact, and where the map $h$ is nontrivial.
In particular note that $p$ gives rise to an isomorphism $M/f(\L)\cong H_1^\psi(S^3\sm L;\L)$,
and that $g$ gives rise to an epimorphism
\be \label{equ1}  H_1^\psi(S^3\sm  L;\L)\cong M/f(\L)\to H_1^\psi(S^3\sm ( J\sqcup K);\L)\,/\,(g\circ f)(\L).\ee
Next we will prove the following claim.

\begin{claim}
The map
\[ H_1^\psi(S^3\sm J;\L)\oplus H_1^\psi(S^3\sm K;\L)\to  H_1^\psi(S^3\sm J\sqcup K;\L)\,/\,(g\circ f)(\L)\]
is a monomorphism.
\end{claim}

We consider the following commutative diagram
\[ \xymatrix @C-0.5cm{
&&\L\ar[d]^-{g\circ f} &  \\
0\ar[r]&H_1^\psi(S^3\sm J;\L)\oplus H_1^\psi(S^3\sm K;\L)\ar@{^{(}->}[r]& H_1^\psi(S^3\sm (J\sqcup K);\L)\ar[r]^-{h}\ar[d]& \L\ar[d]\\
&& \frac{H_1^\psi(S^3\sm (J\sqcup K);\L)}{(g\circ f)(\L)} \ar[r]^-{h} & \L/(h\circ g\circ f)(\L),}\]
where the bottom vertical maps are the obvious projection maps. Furthermore, as above the map $h$ in the middle sequence is nontrivial.

We first note that the bottom left group is $\L$-torsion.
Indeed, in the discussion preceding the proof we saw that $\Delta_L(s,t)\ne 0$.
This implies that the homology group $H_1^\psi(S^3\sm  L;\L)$ is $\L$-torsion.
But by (\ref{equ1}) this also implies that the bottom left group of the diagram is $\L$-torsion.

It follows that in the square the composition of maps given by going down and then right factors through a $\L$-torsion group.
On the other hand we have seen that the map $h\colon H_1^\psi(S^3\sm (J\sqcup K);\L)\to \L$ is nontrivial.
By the commutativity of the square and by the fact that the down-right composition of maps factors through a $\L$-torsion group it now follows that the projection map $\L\to \L/(h\circ g\circ f)(\L)$ cannot be an isomorphism.  But this just means that the composition
\[ \L\xrightarrow{g\circ f} H_1^\psi(S^3\sm (J\sqcup K);\L)\xrightarrow{h}\L\]
is nontrivial, and in particular injective. Put differently, we have
\[ \im\Big(g\circ f\colon \L\to H_1^\psi(S^3\sm (J\sqcup K);\L)\Big)\,\cap\, \ker(h)=0.\]
By the exactness of the middle horizontal sequence we thus see that the intersection
of the images of $(g\circ f)(\L)$ and of $H_1^\psi(S^3\sm J;\L)\oplus H_1^\psi(S^3\sm K;\L)$ in $ H_1^\psi(S^3\sm (J\sqcup K);\L))$ is trivial.
It follows that the map
\[ H_1^\psi(S^3\sm J;\L)\oplus H_1^\psi(S^3\sm K;\L)\to  H_1^\psi(S^3\sm J\sqcup K;\L)/(g\circ f)(\L)\]
is indeed a monomorphism.
This concludes the proof of the claim.

Before we continue with the proof we recall that if
\[ 0\to A\to B\to C\to 0 \]
is a short exact sequence of $\L$-modules, then by \cite[Part~1.3.3]{Hi02} the orders of the modules are related by the following  equality
\be \label{equ:ordmult} \ord_{\L}(B)=\ord_{\L}(A)\cdot \ord_{\L}(C).\ee

We thus see, from the claim and from  (\ref{equ:ordmult}), (\ref{equ:deltaj}) and (\ref{equ:deltak}) that
\[ \Delta_J(s)\cdot \Delta_K(t)\, \,|\, \,\ord_{\L}\left(H_1^\psi(S^3\sm J\sqcup K;\L)/(g\circ f)(\L)\right).\]
On the other hand it follows from (\ref{equ1}) and again from (\ref{equ:ordmult}) that
\[  \ord_{\L}\left(H_1^\psi(S^3\sm (J\sqcup K);\L)/(g\circ
  f)(\L)\right)\,\,\Big{|}\,\, \ord_{\L}\left(H_1^\psi(S^3\sm
  L;\L)\right).
\qedhere
\]
\end{proof}

\subsection{An example of the Alexander polynomial technique}
\label{section:splitting-no-examples}

We consider the oriented link $L=K\,\sqcup\, J=L9a29$ from Figure~\ref{figure:9a29}.
\begin{figure}[h]
  \labellist
  \small\hair 0mm
  \pinlabel {$J$} at 144 100
  \pinlabel {$K$} at 20 88
  \endlabellist
\begin{center}
\includegraphics{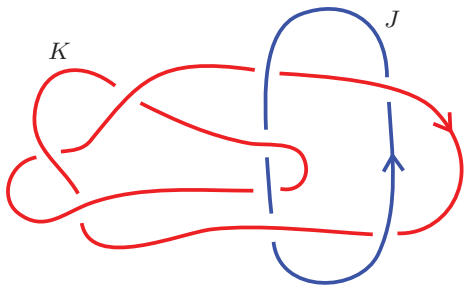}
\end{center}
\caption{The link $L9a29$, with splitting number~$3$.}\label{figure:9a29}
\end{figure}
It has linking number one and it is not hard to see that one can turn it into a split link using three crossing changes between the two components.
The multivariable Alexander polynomial of $L$ is
\[ \ba{rcl} \Delta_L(s,t)&=&s-s^2+t-st+s^2 t-t^2+s t^2-s^2 t^2+t^3-s t^3+s^2 t^3-t^4+s t^4.\ea \]
It is straightforward to see that $\Delta_J(s)\cdot \Delta_K(t)=1-t+t^2$ does not divide $\Delta_L(s,t)$.
It thus follows from Theorem~\ref{prop:split1} that the splitting number of $L$ is three.

This is one of the instances of the use of the Alexander polynomial
which is cited in Section~\ref{section:table-splitting-no-results}, in
Table~\ref{table:splitting-numbers}
(method~\ref{aaitem:alexander-polynomial}).  The other computations
listed in that table as using this method are performed in a similar
fashion; see the LinkInfo tables~\cite{linkinfo} for the
multivariable Alexander polynomials of the other 9 crossing links,
which are $L9a24$, $L9n13$, $L9n14$ and $L9n17$.  Since these are two
component links of linking number 1, the Alexander polynomials of the
components can be obtained by substituting either $t=1$ or $s=1$ into
the multivariable Alexander polynomial in $\Z[s^{\pm 1},t^{\pm 1}]$.

\section{The examples of Batson and Seed}\label{section:batson-seed-examples}

In~\cite{batson-seed}, Batson and Seed constructed a spectral sequence from
the Khovanov homology of a link $L$ to the Khovanov homology of the
split link with the same components as $L$.  This spectral sequence
gives rise to a lower bound on the splitting number, given by the
lowest page on which their spectral sequence collapses.

Batson and Seed computed the lower bound for all links up to 12
crossings and they showed that it provides more information than basic
linking number observations (see our
Lemma~\ref{lem:splittingnumberbasics}) for 17 links.  The lower bound they computed will be denoted by
$b(L)$.  One of the 17 links is a
3-component link with 12 crossings, for which $b(L)=3$ while the
sum of the absolute values of the linking numbers is one.  The remaining 16 links have 2-components
and satisfy $\lk(L)=\pm 1$ and $b(L)=3$.  One of these has 11
crossings, and 15 of these have 12 crossings.  Batson and Seed determined the
splitting numbers for 7 links among these 17 links, while for the other 10
cases the splitting numbers are listed as being either 3 or 5.  This information is given
in~\cite[Table~3]{batson-seed}.

In this section we revisit these links to reprove or improve the
results in~\cite{batson-seed}.  In particular we completely determine
the splitting numbers by using our methods.

\subsection{Using the Alexander polynomial}\label{section:alex-poly-batson-seed}

We first apply our Alexander polynomial method to the examples of~\cite{batson-seed} with at least one knotted component.  This will reprove their splitting number results for these links.    Before we turn to the links of
\cite[Table~3]{batson-seed}, we will discuss a link with
13-crossings in detail, which is also discussed in~\cite{batson-seed}.

\subsubsection*{A 13-crossing example}

Consider the 2-component link $L$ shown
in Figure~\ref{figure:rubermans-link}.  It is the link denoted by
${}^2n_{8862}^{13}$ in~\cite{batson-seed}.
\begin{figure}[H]
\begin{center}
\includegraphics[scale=.65]{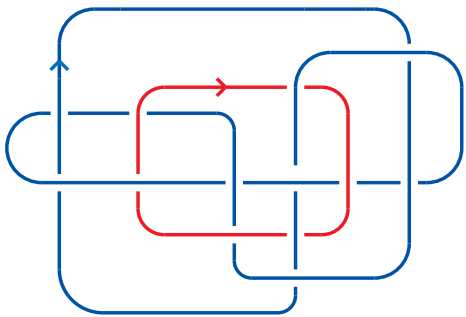}
\end{center}
\caption{A 2-component link of linking number 1 whose splitting number equals 3.}
\label{figure:rubermans-link}
\end{figure}
Note that one component is an unknot and the other is a trefoil.  We
refer to the unknotted component as $J$ and to the knotted component
as~$K$.  It is not hard to see that $L$ can be turned into a split
link using three crossing changes.  On the other hand the linking
number is $1$, so it follows from Lemma
\ref{lem:splittingnumberbasics} that the splitting number is either
one or three.  The invariant $b(L)$ shows that the splitting
number of $L$ is in fact three.

We will now use Theorem~\ref{prop:split1} to give another proof that
the splitting number of $L$ equals three.  We used Kodama's program
\emph{knotGTK} to show that
\[
\begin{split}
\Delta_L(s,t)&=
-s^8t^4+s^7t^5+4s^8t^3-5s^6t^5-6s^8t^2-9s^7t^3+13s^6t^4+11s^5t^5+4s^8t\\
&\quad +17s^7t^2-6s^6t^3-37s^5t^4-14s^4t^5-s^8-12s^7t-10s^6t^2+45s^5t^3-4s^6 \\
&\quad +52s^4t^4+11s^3 t^5+3s^7+12s^6t-24s^5t^2-74s^4t^3-44s^3t^4-5s^2t^5 \\
&\quad +2s^5t +51s^4t^2+67s^3t^3+23s^2t^4+st^5+3s^5-13s^4t-46s^3t^2-39s^2t^3\\
&\quad -7st^4-s^4+11s^3t+25s^2t^2+15s t^3+t^4-3s^2t-9st^2-3t^3+2t^2.
\end{split}
\]
It is straightforward to see
(we used Maple) that
\[
\Delta_L(s,1)\cdot\Delta_L(1,t) = \Delta_J(s)\cdot \Delta_K(t)=1-t+t^2
\]
does not divide $\Delta_L(s,t)$.
Thus it follows from Theorem~\ref{prop:split1} that the splitting number of $L$ is not one.
By the above observations we therefore see that the splitting number of $L$ is equal to three.

\subsubsection*{Seven 12-crossing examples}

In~\cite[Table~3]{batson-seed}, Batson and Seed give
seven examples of 2-component 12 crossing links which have linking
number equal to one and for which $b(L)$ detects that the
splitting number is three.

In Table~\ref{table:seven-12-crossing-links-and-alexander-poly}, we
list the links together with their Dowker-Thistlethwaite (DT) codes
and multivariable Alexander polynomials.  The translation between the names we use
(following LinkInfo~\cite{linkinfo}) and the convention used in
\cite{batson-seed} is given by $L12nX = {}^2a^{12}_{X+4196}$.  All
these Alexander polynomials are irreducible.  For each link, both components are trefoils, so $\Delta_L(s,1) = 1-s+s^2$ and $\Delta_L(1,t) = 1-t+t^2$ do not divide the multivariable Alexander polynomial.  Thus it follows from Theorem~\ref{prop:split1} that
the splitting number of each of these links is at least three, which
recovers the results of Batson and Seed.  Inspection of the diagrams
shows that the splitting numbers are indeed equal to 3.

\begin{table}
\begin{center}
\small
\begin{tabular}{lm{33mm}m{75mm}<{\raggedright}}
\toprule
Link & DT code & Alexander polynomial\\
\midrule
$L12n1342$ & $(14, \ol{6}, \ol{10}, 16, \ol{4}, \ol{18})$,\linebreak\hbox{}\hfill$(\ol{20}, 22, 8, \ol{24}, \ol{2}, 12)$ &
$s^2t^4-st^4-s^2t^2+s^2t+st^2+t^3-t^2-s+1$ \\
\noalign{\vspace{.7ex}}
$L12n1350$ & $(14, \ol{6}, \ol{10}, \ol{16}, \ol{4}, 20)$,\linebreak\hbox{}\hfill$(12, 22, \ol{8}, 24, 2, 18)$ &
$-s^4t^3+s^3t^4+s^4t^2+s^3t^3-2s^2t^4-2s^3t^2+s^2t^3$\linebreak\hbox{}\hfill$\hbox{}+st^4+s^2t^2+s^3+s^2t-2st^2-2s^2+st+t^2+s-t$\\
\noalign{\vspace{.7ex}}
$L12n1357$ & $(14, \ol{6}, \ol{10}, 16, \ol{4}, 22)$,\linebreak\hbox{}\hfill$(20, 2, 8, 24, 12, 18)$ &
$-s^4t^4+s^4t^3+s^3t^4-s^4t^2+s^4t-s^2t^3$\linebreak\hbox{}\hfill$\hbox{}+s^2t^2-s^2t+t^3-t^2+s+t-1$\\
\noalign{\vspace{.7ex}}
$L12n1363$ & $(14, \ol{6}, \ol{10}, \ol{18}, \ol{4}, 22)$,\linebreak\hbox{}\hfill$(2, 20, \ol{8}, 24, 12, 16)$ &
$2s^2t^2-3s^2t-3st^2+2s^2+5st+2t^2-3s-3t+2$\\
\noalign{\vspace{.7ex}}
$L12n1367$ & $(14, \ol{6}, \ol{10}, \ol{18}, \ol{4}, 24)$,\linebreak\hbox{}\hfill$(2, 12, 22, \ol{8}, 16, 20)$ &
$s^4t^2+s^3t^3-s^2t^4-s^4t-2s^3t^2+st^4+s^3t$\linebreak\hbox{}\hfill$\hbox{}+s^2t^2+st^3+s^3-2st^2-t^3-s^2+st+t^2$\\
\noalign{\vspace{.7ex}}
$L12n1374$ & $(14, \ol{6}, \ol{10}, 20, \ol{4}, 16)$,\linebreak\hbox{}\hfill$(2, 12, 22, 24, 8, 18)$ &
$s^4t^3+s^3t^4-s^4t^2-s^3t^3-s^2t^4+s^4t+st^4$\linebreak\hbox{}\hfill$\hbox{}-s^4+s^2t^2-t^4+s^3+t^3-s^2-st-t^2+s+t$\\
\noalign{\vspace{.7ex}}
$L12n1404$ & $(14, \ol{8}, \ol{18}, \ol{12}, \ol{22}, \ol{4})$,\linebreak\hbox{}\hfill$(20, 2, 24, \ol{6}, 16, \ol{10})$ &
$2s^2t^3-st^4-2s^2t^2-st^3+t^4$\linebreak\hbox{}\hfill$\hbox{}+3st^2+s^2-st-2t^2-s+2t$\\
\bottomrule
\end{tabular}
\medskip
\end{center}
\caption{Seven 12-crossing links and their Alexander polynomials}
\label{table:seven-12-crossing-links-and-alexander-poly}
\end{table}

\subsection{Using the covering link technique}\label{section:batson-seed-covering-link-technique}

Batson and Seed, in~\cite[Table~3]{batson-seed}, give nine further
examples of links which have two unknotted components and linking
number $\pm 1$.  They list these links as having splitting number
either three or five.  Translating notation again, we have: $L11a372 =
{}^2a^{11}_{739}$, $L12aX = {}^2a^{12}_{X+1288}$ and $L12nY =
{}^2n^{12}_{Y+4196}$.

Table~\ref{table:nine-links-and-covering-slice-genus} lists
the results of our computations, giving the slice genus of the knot
obtained by taking a $2$-fold branched cover of $S^3$, branched over
one of the components, the method which we use to compute the slice
genus, and the resulting splitting number obtained by the methods of
Section~\ref{section:covering-link-technique}.

\begin{table}
\begin{center}
\small
\def\lineseparator{\noalign{\vspace{.7ex}}}
\begin{tabular}{lm{35mm}@{\hspace{.5em}}c@{\hspace{.5em}}cc}
\toprule
Link & DT code &
{\begin{tabular}{@{}c@{}}Covering knot\\slice genus\end{tabular}} &
{\begin{tabular}{@{}c@{}}Method for\\slice genus\end{tabular}} &
{\begin{tabular}{@{}c@{}}Splitting\\number\end{tabular}} \\
\midrule
$L11a372$ & $(12, 14, 16, 20, 18)$,\linebreak\hbox{}\hfill$(10, 2, 4, 22, 8, 6)$ & 2 & $\sigma=-4$ & 5 \\
\lineseparator
$L12a1233$ & $(12, 14, 16, 18, 20)$,\linebreak\hbox{}\hfill$(2, 24, 4, 6, 22, 8, 10)$ & 2 & $\sigma=-4$ & 5 \\
\lineseparator
$L12a1264$ & $(12, 14, 16, 20, 18)$,\linebreak\hbox{}\hfill$(2, 24, 4, 22, 8, 6, 10)$ & 2 & $\sigma=-4$ & 5  \\
\lineseparator
$L12a1384$ & $(14, 8, 16, 24, 18, 20)$,\linebreak\hbox{}\hfill$(2, 22, 4, 10, 12, 6)$ & 2 & $\sigma=-4$ & 5 \\
\lineseparator
$L12n1319$ & $(12, 14, 16, 24, \ol{18})$,\linebreak\hbox{}\hfill$(2, 10, 22, \ol{20}, \ol{8}, 4, 6)$ & 1 & $\sigma=-2$ & 3 \\
\lineseparator
$L12n1320$ & $(12, \ol{14}, \ol{16}, 24, \ol{18})$,\linebreak\hbox{}\hfill$(6, 22, \ol{20}, \ol{8}, \ol{4}, \ol{2}, 10)$ & 1 & {\begin{tabular}{@{}>{$}c<{$}@{}}p=3,\, q=7,\\\Delta^{\chi}(s) = 7s^2+15s+7\end{tabular}} & 3\\
\lineseparator
$L12n1321$ & $(12, 14, 16, 24, \ol{18})$,\linebreak\hbox{}\hfill$(10, 2, 22, \ol{20}, \ol{8}, 4, 6)$ & 2 & $s=4$ & 5\\
\lineseparator
$L12n1323$ & $(12, 14, 18, 16, \ol{20})$,\linebreak\hbox{}\hfill$(10, 2, 24, 6, \ol{22}, \ol{8}, 4)$ & 2 & $\sigma=-4$ & 5\\
\lineseparator
$L12n1326$ & $(12, 14, \ol{18}, 24, \ol{20})$,\linebreak\hbox{}\hfill$(8, 2, 4, \ol{22}, \ol{10}, \ol{6}, \ol{16})$ & 1 & {\begin{tabular}{@{}>{$}c<{$}@{}}p=3,\, q=7,\\ \Delta^{\chi}(s) = 7s^2-71s+7\end{tabular}} & 3 \\
\bottomrule
\end{tabular}
\medskip
\end{center}
\caption{Nine links and their covering knot slice genus}
\label{table:nine-links-and-covering-slice-genus}
\end{table}

The methods we use to compute the slice genus of the covering knot are
as follows.  First, the slice genus of a knot is bounded below by half
the absolute value of its signature $\sigma(K)=\sign(A+A^T)$, where
$A$ is a Seifert matrix of~$K$, by \cite[Theorem~9.1]{Mu67}.  We used a Python software package of
the first author to compute~$\sigma(K)$.  The Rasmussen
$s$-invariant~\cite{Rasmussen-s-invariant} also gives a lower bound by
$|s(K)| \leq 2g_4(K)$.  We used \texttt{JavaKh} of J. Green and
S. Morrison to compute~$s(K)$.

We can also prove that a knot is not
slice using the twisted Alexander polynomial~\cite{KirkLivingston},
denoted $\Delta_K^{\chi}(s) \in \Q(\zeta_q)[s^{\pm 1}]$, where
$\zeta_q$ is a primitive $q$-th root of unity, associated to the
$p$-fold cyclic cover $X_p$ of the knot exterior $X$ and a character
$\chi \colon TH_1(X_p;\Z) \to \Z_q$.  For slice knots, there exists a metaboliser of the $\Q/\Z$ valued
linking form on $H_1(X_p;\Z)$, such that for characters $\chi$ which vanish on the metaboliser, the twisted Alexander polynomial
factorises (up to a unit) as $g(s)\ol{g(s)}$, for some $g\in \Q(\zeta_q)[s^{\pm 1}]$.  By
checking that this condition does not hold for all metabolisers, we
can prove that a knot is not slice.  (For each covering knot $K$ to
which we apply this, all metabolisers give the
same polynomial~$\Delta^\chi_K$.) Our computations of twisted Alexander polynomials
were performed using a Maple program written by C.~Herald, P.~Kirk and
C.~Livingston~\cite{HKL}.

The invariants discussed above give us lower bounds on the slice
genera of the covering knots.  We do not need to know the precise slice genera in order to obtain lower bounds.  Nevertheless we point out that we are able to determine them. In each case we found the requisite crossing changes to
split the link, so an application of
Theorem~\ref{theorem:covering-link-lk-one-ribbon-intro} gives us an upper bound on the slice genus of the
covering knot, which implies that the above lower bounds are sharp.

For the links $L11a372$, $L12a1233$, $L12a1264$, $L12a1384$,
$L12n1321$ and $L12n1323$, we are able to show that the splitting
numbers of these links is~5.  Amusingly we use Khovanov homology, in
the guise of the $s$-invariant, to compute that the slice genus of the
covering knot of $L12n1321$ is~2.  We remark that this knot has
$\sigma=-2$, which is only sufficient to show that the splitting
number is at least~3.

For the other links, as in
Section~\ref{section:alex-poly-batson-seed}, our obstruction gives the
same information as the Batson-Seed lower bound, namely that the
splitting number is at least three.  For the links $L12n1319$,
$L12n1320$ and $L12n1326$ we looked at the diagrams and found the
crossing changes to verify that the splitting number is indeed 3.

We present one example in detail, the link $L11a372$, which is shown as given by LinkInfo on
the left of Figure~\ref{figure:L11a372}, while on the right the link $L11a372$ is shown after an isotopy, to
prepare for making a diagram of a covering link.  It is easy to see
from the diagram that the splitting number is at most 5.  The link
$L11a372$ has $b(L11a372)=3$.

\begin{figure}[H]
\includegraphics[scale=.65]{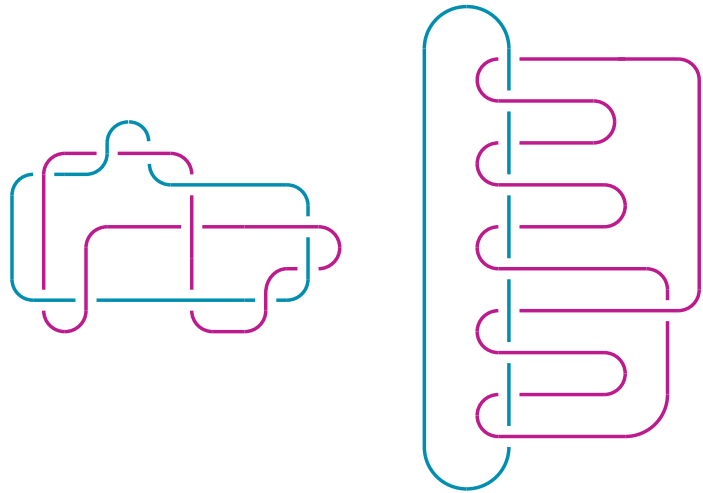}
\caption{Left: the link $L11a372$. Right: after an isotopy.}
\label{figure:L11a372}
\end{figure}

The 2-fold covering link obtained by branching over the left hand
component is shown in Figure~\ref{figure:L11a372-2}.  This turns out
to be the knot $7_5$, which according to KnotInfo~\cite{knotinfo} has
$|\sigma|= 4$ and slice genus $2$.  Therefore by
Theorem~\ref{theorem:covering-link-lk-one-ribbon-intro}, the splitting
number is 5.

\begin{figure}[H]
\includegraphics[scale=.6]{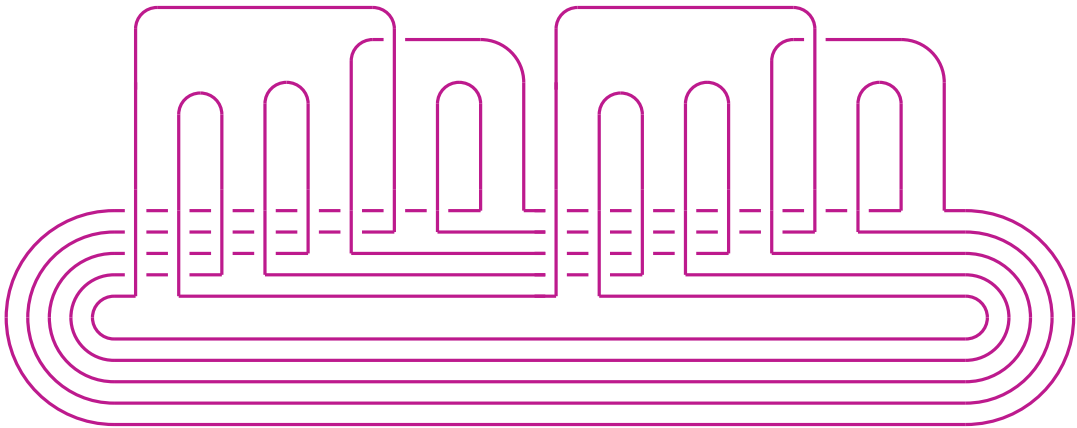}
\caption{The covering link obtained by taking a 2-fold branched
  covering over the left hand component of the link in
  Figure~\ref{figure:L11a372}.}
\label{figure:L11a372-2}
\end{figure}

\subsection{A 3-component example}\label{section:batson-seed-3-comp-example}

There is one final
link listed in~\cite[Table~3]{batson-seed} as having splitting number
either three or five, namely the 3-component link $L:= L12a1622$,
which is shown in Figure~\ref{figure:L12a1622-1}.  In the notation
of~\cite{batson-seed}, $L$ is the link ${}^3a^{12}_{2910}$.

\begin{figure}[H]
  \labellist
  \small\hair 0mm
  \pinlabel {$L_1$} at 0 110
  \pinlabel {$L_2$} at 140 273
  \pinlabel {$L_3$} at 233 18
  \endlabellist
  \includegraphics[scale=.65]{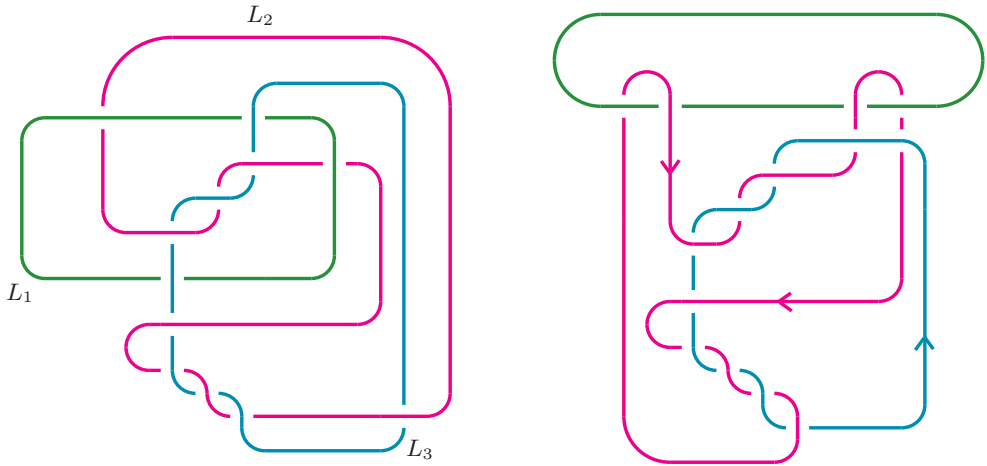}
  \caption{Left: the link $L12a1622$.  Right: the same link, after an
    isotopy to prepare for taking a covering link by branching over the
    top component.}
  \label{figure:L12a1622-1}
\end{figure}

\begin{figure}[H]
\begin{center}
\includegraphics[scale=.5]{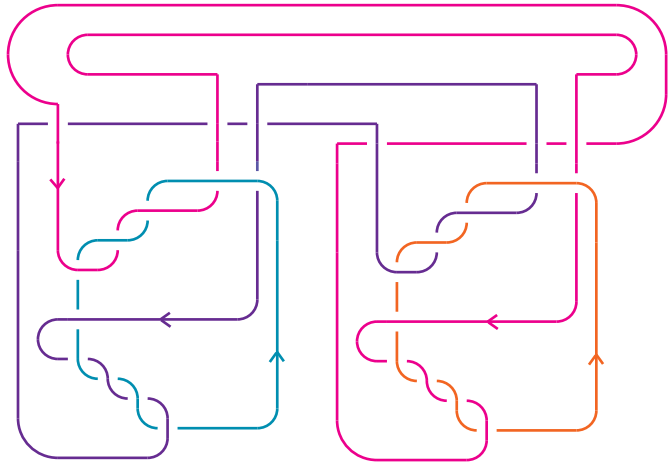}
\end{center}
\caption{The 2-fold covering link of the link in Figure~\ref{figure:L12a1622-1}, branched over $L_1$.}
\label{figure:L12a1622-2}
\end{figure}

  We show that the splitting number of $L$ is in fact $5$.  Note that the
components are unknotted, and the only nonzero linking number is
between $L_2$ and $L_3$, which have $|\lk(L_2,L_3)| = 1$.  Thus the
splitting number is odd by Lemma~\ref{lem:splittingnumberbasics}.  It
is easy to find $5$ crossing changes which suffice.

We begin by showing that three crossing changes involving just $L_2$
and $L_3$ do not suffice to split the link.  We take the 2-fold
covering link $J$ with respect to~$L_1$.  The result of an isotopy to
prepare for taking such a covering is shown on the right of
Figure~\ref{figure:L12a1622-1}.  The resulting covering link $J$ is
shown in Figure~\ref{figure:L12a1622-2}.  The link $J$ has splitting
number $10$ by Lemma~\ref{lem:splittingnumberbasics}, with a sharp lower bound given by the sum of the absolute
values of the linking numbers between the components.  By
Corollary~\ref{corollary:covering-link-zero-lk-no}, we have that
$\sp_1(L) \geq 5$.

Combining $\sp_1(L) \geq 5$ with the linking number, it follows that
if $\sp(L) \le 3$, then exactly one crossing change involving
$(L_2,L_3)$ is required to split the link, and there can be either two
additional $(L_1, L_2)$ crossing changes, or two $(L_1, L_3)$ crossing
changes.  We will give the argument to show that the first possibility
cannot happen; the argument discounting the second possibility is
analogous.

Suppose that two $(L_1,L_2)$ crossing changes and one $(L_2,L_3)$
crossing change yields the unlink.  Applying
Corollary~\ref{cor:covering-genus-lower-bound} (with $m=3$,
$\a=2$, $\b=1$, $g_4(L_k)=0$), it follows that the covering link $J\subset
S^3$ bounds an oriented surface $F$ of Euler characteristic
$2(3-1)-2-4 = -2$ which is smoothly embedded in~$D^4$ and has no
closed component.  Also, $F$ is connected by the last part of
Corollary~\ref{cor:covering-genus-lower-bound} since both
$L_1$ and $L_3$ are involved in some crossing change
with~$L_2$.
Since $J$ has 4 components, $F$ is a 3-punctured disc.  That is, $J$
is weakly slice.

To show that this cannot be the case for $J$, we use the link signature
invariant, which is defined similarly to the knot signature: for a link $J$, choose a surface $V$ in $S^3$ bounded by $J$ ($V$ may be
disconnected), define the Seifert pairing on $H_1(V)$ and an
associated Seifert matrix $A$ as usual.  Then the \emph{link
  signature} of $J$ is defined by $\sigma(J) = \sign (A+A^T)$.  Due to
K.~Murasugi~\cite{Mu67}, if an $m$-component link $J$ bounds a smoothly
embedded oriented surface $F$ in $D^4$, we have $|\sigma(J)| \le
2g(F)+m-b_0(F)$ where $g(F)$ is the genus and $b_0(F)$ is the 0th
Betti number of $F$.  For our covering link $J$, since it bounds a
3-punctured disc in $D^4$, we have $|\sigma(J)|\le 3$.  Here we orient
$J$ as in Figure~\ref{figure:L12a1622-2}; this orientation is obtained using the orientations of
$L_2$ and $L_3$ shown on the right of
Figure~\ref{figure:L12a1622-1}.  On the other hand, a computation aided by a Python
software package of the first author shows that $\sigma(J)=-7$.  From this
contradiction it follows that one $(L_2, L_3)$ crossing change and two
$(L_1, L_2)$ crossing changes never split~$L$.

\section{Links with 9 or fewer crossings}\label{section:table-splitting-no-results}

In Table~\ref{table:splitting-numbers} we give the splitting numbers for the
links of 9 crossings or fewer, together with the method which is used to give a
sharp lower bound for the splitting number.  The entry in the method
column of the table refers to the list below.

In the case of 2-component links with unknotted components and linking number one, knowing that the \emph{unlinking number} is greater than one implies that the splitting number is at least three.
Recall that by definition the unlinking number of an $m$-component link $L$ is the minimal number of crossing changes required to convert $L$ to the $m$-component unlink.  Note that for this link invariant, crossing changes of a component with itself are in general permitted.

In Method~\ref{aaitem:two-comp-lk-one-unknotted} below, we will make use of computations of unlinking numbers made by P.~Kohn in~\cite{Kohn93}, where making considerable use of his earlier work in~\cite{Kohn91}, he computed the unlinking numbers of 2-component links
with 9 or fewer crossings, in all but 5 cases.

\begin{table}[t]
\begin{center}
\small
\begin{tabular}[t]{lcc}
\toprule
Link $L$ & $\sp(L)$ & Method\\
\midrule
$L2a1$	&1&(\ref{aaitem:non-trivial-link})\\
$L4a1$	&2&(\ref{aaitem:non-trivial-link})\\
$L5a1$	&2&(\ref{aaitem:non-trivial-link})\\
$L6a1$  &2&(\ref{aaitem:non-trivial-link})\\
$L6a2$	&3&(\ref{aaitem:non-trivial-link})\\
$L6a3$	&3&(\ref{aaitem:non-trivial-link})\\
$L6a4$	&2&(\ref{aaitem:non-trivial-link})\\
$L6a5$	&3&(\ref{aaitem:non-trivial-link})\\
$L6n1$	&3&(\ref{aaitem:non-trivial-link})\\
$L7a1$	&2&(\ref{aaitem:non-trivial-link})\\
$L7a2$	&2&(\ref{aaitem:non-trivial-link})\\
$L7a3$	&2&(\ref{aaitem:non-trivial-link})\\
$L7a4$	&2&(\ref{aaitem:non-trivial-link})\\
$L7a5$	&1&(\ref{aaitem:non-trivial-link})\\
$L7a6$	&3&(\ref{aaitem:two-comp-lk-one-unknotted})\\
$L7a7$	&3&(\ref{aaitem:non-trivial-link})\\
$L7n1$	&2&(\ref{aaitem:non-trivial-link})\\
$L7n2$	&2&(\ref{aaitem:non-trivial-link})\\
$L8a1$	&2&(\ref{aaitem:non-trivial-link})\\
$L8a2$	&2&(\ref{aaitem:non-trivial-link})\\
$L8a3$	&2&(\ref{aaitem:non-trivial-link})\\
$L8a4$	&2&(\ref{aaitem:non-trivial-link})\\
$L8a5$	&2&(\ref{aaitem:non-trivial-link})\\
$L8a6$	&2&(\ref{aaitem:non-trivial-link})\\
$L8a7$	&2&(\ref{aaitem:non-trivial-link})\\
$L8a8$	&3&(\ref{aaitem:two-comp-lk-one-unknotted})\\
$L8a9$	&3&(\ref{aaitem:two-comp-lk-one-unknotted})\\
$L8a10$	&3&(\ref{aaitem:non-trivial-link})\\
$L8a11$	&3&(\ref{aaitem:non-trivial-link})\\
$L8a12$	&4&(\ref{aaitem:non-trivial-link})\\
$L8a13$	&4&(\ref{aaitem:non-trivial-link})\\
$L8a14$	&4&(\ref{aaitem:non-trivial-link})\\
$L8a15$	&3&(\ref{aaitem:non-trivial-link})\\
$L8a16$	&3&(\ref{aaitem:two-comp-lk-one-unknotted})\\
$L8a17$	&4&(\ref{aaitem:non-trivial-link})\\
$L8a18$	&4&(\ref{aaitem:non-trivial-link})\\
$L8a19$	&2&(\ref{aaitem:non-trivial-link})\\
$L8a20$	&4&(\ref{aaitem:non-trivial-link})\\
$L8a21$	&4&(\ref{aaitem:non-trivial-link})\\
$L8n1$	&2&(\ref{aaitem:non-trivial-link})\\
$L8n2$	&2&(\ref{aaitem:non-trivial-link})\\
$L8n3$	&4&(\ref{aaitem:non-trivial-link})\\
$L8n4$	&4&(\ref{aaitem:non-trivial-link})\\
$L8n5$	&2&(\ref{aaitem:non-trivial-link})\\
\bottomrule
\end{tabular}%
\qquad
\begin{tabular}[t]{lcc}
\toprule
Link $L$ & $\sp(L)$ & Method\\
\midrule
$L8n6$	&4&(\ref{aaitem:non-trivial-link})\\
$L8n7$	&4&(\ref{aaitem:non-trivial-link})\\
$L8n8$	&4&(\ref{aaitem:non-trivial-link})\\
$L9a1$	&2&(\ref{aaitem:non-trivial-link})\\
$L9a2$	&2&(\ref{aaitem:non-trivial-link})\\
$L9a3$	&2&(\ref{aaitem:non-trivial-link})\\
$L9a4$	&2&(\ref{aaitem:non-trivial-link})\\
$L9a5$	&2&(\ref{aaitem:non-trivial-link})\\
$L9a6$	&2&(\ref{aaitem:non-trivial-link})\\
$L9a7$	&2&(\ref{aaitem:non-trivial-link})\\
$L9a8$	&2&(\ref{aaitem:non-trivial-link})\\
$L9a9$	&2&(\ref{aaitem:non-trivial-link})\\
$L9a10$	&2&(\ref{aaitem:non-trivial-link})\\
$L9a11$	&2&(\ref{aaitem:non-trivial-link})\\
$L9a12$	&2&(\ref{aaitem:non-trivial-link})\\
$L9a13$ &2&(\ref{aaitem:non-trivial-link})\\
$L9a14$	&2&(\ref{aaitem:non-trivial-link})\\
$L9a15$	&2&(\ref{aaitem:non-trivial-link})\\
$L9a16$	&2&(\ref{aaitem:non-trivial-link})\\
$L9a17$	&2&(\ref{aaitem:non-trivial-link})\\
$L9a18$	&2&(\ref{aaitem:non-trivial-link})\\
$L9a19$	&2&(\ref{aaitem:non-trivial-link})\\
$L9a20$	&3&(\ref{aaitem:two-comp-lk-one-unknotted})\\
$L9a21$	&1&(\ref{aaitem:non-trivial-link})\\
$L9a22$	&3&(\ref{aaitem:two-comp-lk-one-unknotted})\\
$L9a23$	&3&(\ref{aaitem:non-trivial-link})\\
$L9a24$	&3&(\ref{aaitem:alexander-polynomial})\\
$L9a25$	&3&(\ref{aaitem:non-trivial-link})\\
$L9a26$	&3&(\ref{aaitem:two-comp-lk-one-unknotted})\\
$L9a27$	&1&(\ref{aaitem:non-trivial-link})\\
$L9a28$	&3&(\ref{aaitem:non-trivial-link})\\
$L9a29$	&3&(\ref{aaitem:alexander-polynomial})\\
$L9a30$	&3&(\ref{aaitem:two-comp-lk-one-unknotted})\\
$L9a31$	&1&(\ref{aaitem:non-trivial-link})\\
$L9a32$	&3&(\ref{aaitem:non-trivial-link})\\
$L9a33$	&3&(\ref{aaitem:non-trivial-link})\\
$L9a34$	&2&(\ref{aaitem:non-trivial-link})\\
$L9a35$	&2&(\ref{aaitem:non-trivial-link})\\
$L9a36$	&4&(\ref{aaitem:two-comp-lk-zero-unknotted})\\
$L9a37$	&2&(\ref{aaitem:non-trivial-link})\\
$L9a38$	&2&(\ref{aaitem:non-trivial-link})\\
$L9a39$	&2&(\ref{aaitem:non-trivial-link})\\
$L9a40$	&4&(\ref{aaitem:two-comp-lk-zero-unknotted})\\
$L9a41$	&2&(\ref{aaitem:non-trivial-link})\\
\bottomrule
\end{tabular}%
\qquad
\begin{tabular}[t]{lcc}
\toprule
Link $L$ & $\sp(L)$ & Method\\
\midrule
$L9a42$	&2&(\ref{aaitem:non-trivial-link})\\
$L9a43$	&3&(\ref{aaitem:non-trivial-link})\\
$L9a44$	&3&(\ref{aaitem:non-trivial-link})\\
$L9a45$	&3&(\ref{aaitem:non-trivial-link})\\
$L9a46$	&3&(\ref{aaitem:two-comp-lk-one-unknotted})\\
$L9a47$ &4&(\ref{aaitem:linking-no-whitehead})\\
$L9a48$	&4&(\ref{aaitem:non-trivial-link})\\
$L9a49$	&4&(\ref{aaitem:non-trivial-link})\\
$L9a50$	&4&(\ref{aaitem:linking-no-whitehead})\\
$L9a51$	&4&(\ref{aaitem:non-trivial-link})\\
$L9a52$	&4&(\ref{aaitem:linking-no-whitehead})\\
$L9a53$	&2&(\ref{aaitem:non-trivial-link})\\
$L9a54$	&4&(\ref{aaitem:linking-no-whitehead})\\
$L9a55$	&4&(\ref{aaitem:non-trivial-link})\\
$L9n1$	&2&(\ref{aaitem:non-trivial-link})\\
$L9n2$	&2&(\ref{aaitem:non-trivial-link})\\
$L9n3$	&2&(\ref{aaitem:non-trivial-link})\\
$L9n4$	&2&(\ref{aaitem:non-trivial-link})\\
$L9n5$	&2&(\ref{aaitem:non-trivial-link})\\
$L9n6$	&2&(\ref{aaitem:non-trivial-link})\\
$L9n7$	&2&(\ref{aaitem:non-trivial-link})\\
$L9n8$	&2&(\ref{aaitem:non-trivial-link})\\
$L9n9$	&2&(\ref{aaitem:non-trivial-link})\\
$L9n10$	&2&(\ref{aaitem:non-trivial-link})\\
$L9n11$	&2&(\ref{aaitem:non-trivial-link})\\
$L9n12$	&2&(\ref{aaitem:non-trivial-link})\\
$L9n13$	&3&(\ref{aaitem:alexander-polynomial})\\
$L9n14$	&3&(\ref{aaitem:alexander-polynomial})\\
$L9n15$	&3&(\ref{aaitem:non-trivial-link})\\
$L9n16$	&3&(\ref{aaitem:non-trivial-link})\\
$L9n17$	&3&(\ref{aaitem:alexander-polynomial})\\
$L9n18$	&4&(\ref{aaitem:non-trivial-link})\\
$L9n19$	&4&(\ref{aaitem:non-trivial-link})\\
$L9n20$	&3&(\ref{aaitem:non-trivial-link})\\
$L9n21$	&3&(\ref{aaitem:non-trivial-link})\\
$L9n22$	&3&(\ref{aaitem:non-trivial-link})\\
$L9n23$	&4&(\ref{aaitem:linking-no-whitehead})\\
$L9n24$	&4&(\ref{aaitem:linking-no-whitehead})\\
$L9n25$	&2&(\ref{aaitem:linking-no-whitehead})\\
$L9n26$	&4&(\ref{aaitem:linking-no-whitehead})\\
$L9n27$	&4&(\ref{aaitem:linking-no-whitehead})\\
$L9n28$	&4&(\ref{aaitem:linking-no-whitehead})\\
{\setbox0\hbox{\hphantom{$L9n28$}}\hbox to\wd0{\hss$\cdot$\hss}} & $\cdot$ & $\cdot$ \\
{\setbox0\hbox{\hphantom{$L9n28$}}\hbox to\wd0{\hss$\cdot$\hss}} & $\cdot$ & $\cdot$ \\
\bottomrule
\end{tabular}
\end{center}
\medskip

\caption{Splitting numbers of links with 9 or fewer crossings.}
\label{table:splitting-numbers}
\end{table}

\begin{enumerate}
\item\label{aaitem:non-trivial-link} Using
  Lemma~\ref{lem:splittingnumberbasics}, the linking numbers determine
  the lower bound for the splitting number, by providing a lower bound
  or by fixing the splitting number modulo~2.
\item\label{aaitem:linking-no-whitehead}  A combination of linking numbers and either one or two
  Whitehead links as a sublink determine a lower bound for the
  splitting number.  That is, Lemma~\ref{lem:splittingnumberbasics}
  provides a sharp lower bound, with $c(L)\neq 0$.
\item\label{aaitem:two-comp-lk-one-unknotted} This is a link where the
  sum of the linking numbers is one and the components are unknotted,
  but which does not have unlinking number one, and so cannot have
  splitting number one.  Therefore the splitting number is at least
  three.

  For the 2-component case (all which use this method have two
  components apart from $L9a46$ and $L8a16$), we know that this link
  does not have unlinking number one by~\cite{Kohn93}.
   Kohn did not explicitly give an argument that the unlinking number of
  $L9a30$ is at least 2, but we computed the splitting
  number of $L9a30$ in Section~\ref{section:exampleII}.

  For the 3-component links $L8a16$ and
  $L9a46$, we show that the splitting number (and unlinking number) is not
  one in Sections~\ref{section:link-L8a16}
  and~\ref{section:link-L9a46} respectively.
\item\label{aaitem:alexander-polynomial} A 2-component link of
  linking number one, with at least one component knotted.  The
  Alexander polynomials of the components do not divide the
  multivariable Alexander polynomial of the link, so by
  Theorem~\ref{prop:split1} the splitting number must be at least 3.
  See Section~\ref{section:splitting-no-examples} for an example of
  this argument in action, for the link $L9a29$.
\item\label{aaitem:two-comp-lk-zero-unknotted} A 2-component link
  with unknotted components and linking number zero modulo two.  For
  the link $L9a36$, we note that the unlinking number is not two by
  \cite{Kohn93}.  Thus the splitting number must be at least 4.  For
  the link $L9a40$, we show that the splitting number is not two in
  Section~\ref{section:link-L9a40}.
\end{enumerate}

We remark that some of the splitting numbers in the table are also
given in~\cite{batson-seed}.

\section{Arguments for the splitting number of particular links}\label{section:arguments-splitting-no-particular-links}

\subsection{The link $L9a40$}\label{section:link-L9a40}

  The link $L9a40$ is shown on the left of Figure~\ref{figure:L9a40}.  We claim
that the splitting number of $L9a40$ is four.  Note that the linking
number is zero, so the splitting number is either two or four, since
it is easy to see from the diagram that four crossing changes suffice
to split the link.

\begin{figure}[H]
\includegraphics[scale=.65]{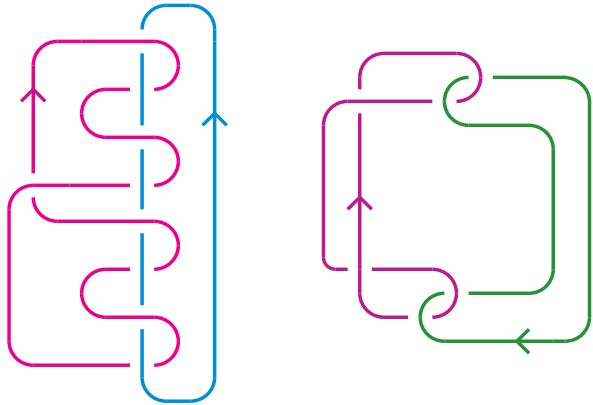}
\caption{Left: the link $L9a40$. Right: the 2-fold branched cover with respect to the right hand component.}
\label{figure:L9a40}
\end{figure}

To show that the splitting number cannot be two, we consider the
2-component link obtained by taking the 2-fold covering link with
respect to the right hand component, which is shown on the right of
Figure~\ref{figure:L9a40}.  This is the link~$L6a1$.  By
Corollary~\ref{corollary:splitting-number-2-covering-link}, if
$\sp(L9a40)=2$, then $L6a1$ would bound an annulus smoothly embedded
in~$D^4$.  Thus, any internal band sum of $L6a1$, which is a knot, would
have slice genus at most one.  But the band sum of $L6a1$ shown in
Figure~\ref{figure:L9a40-two} is the knot $7_5$, which has signature
$4$ and smooth slice genus~$2$.  It follows that the splitting number
of $L9a40$ is four as claimed.

\begin{figure}[H]
\begin{center}
\includegraphics[width=4cm]{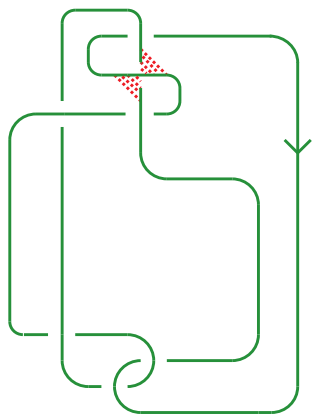}
\end{center}
\caption{A band sum of $L6a1$ to produce the knot $7_5$.  The band which was added is coloured in with dotted lines; it has a half twist in it.}
\label{figure:L9a40-two}
\end{figure}

\subsection{The link $L8a16$}\label{section:link-L8a16}

The link $L8a16$ is shown in Figure~\ref{figure:L8a16-one}.  The
components are labelled $L_1$, $L_2$ and $L_3$.  The linking number
$|\lk(L_1,L_2)|=1$, and the other linking numbers are trivial.  We
claim that $\sp(L8a16)=3$.  It is not hard to find three crossing
changes which work; for example change all three of the crossings
where $L_2$ passes over $L_1$ in Figure~\ref{figure:L8a16-one}.  By
this observation and Lemma~\ref{lem:splittingnumberbasics} the
splitting number is either one or three.  We therefore need to show
that it is not possible to split the link with a single crossing
change.  (We remark that this is the same as showing that the
unlinking number is greater than one, since the components are
unknotted.)

By linking number considerations a single crossing change would have
to involve $L_1$ and~$L_2$.  To discard this eventuality, we will take
a 2-fold covering link branched over $L_3$.

\begin{figure}[H]
  \labellist
  \small\hair 0mm
  \pinlabel {$L_1$} at 8 138
  \pinlabel {$L_2$} at 90 195
  \pinlabel {$L_3$} at 42 25
  \endlabellist
\includegraphics[scale=.6]{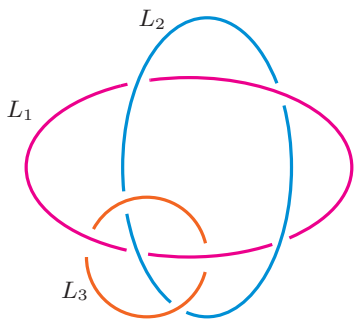}
\caption{The link $L8a16$.}
\label{figure:L8a16-one}
\end{figure}

The left of Figure~\ref{figure:L8a16-cover-over-L3} shows the link
$L8a16$ after an isotopy; the right hand picture shows the $2$-fold
cover branched over~$L_3$.  Call this link~$J$.  The sum of linking
numbers $\sum_{i <j} |\lk(J_i,J_j)| = 6$, so $\sp(J) \ge 6$ by
Lemma~\ref{lem:splittingnumberbasics} (in fact $\sp(J)=u(J)=6$).
Therefore, by Corollary~\ref{corollary:covering-link-zero-lk-no}, we
see that $\sp_3(L8a16) \ge 3$.  (Recall that $\sp_i(L)$ denotes the
splitting number of $L$ where the component $L_i$ is not involved in
any crossing changes.)  Thus, as claimed, it is
not possible to split the link in a single crossing.

\begin{figure}[H]
\begin{center}
\includegraphics[scale=.6]{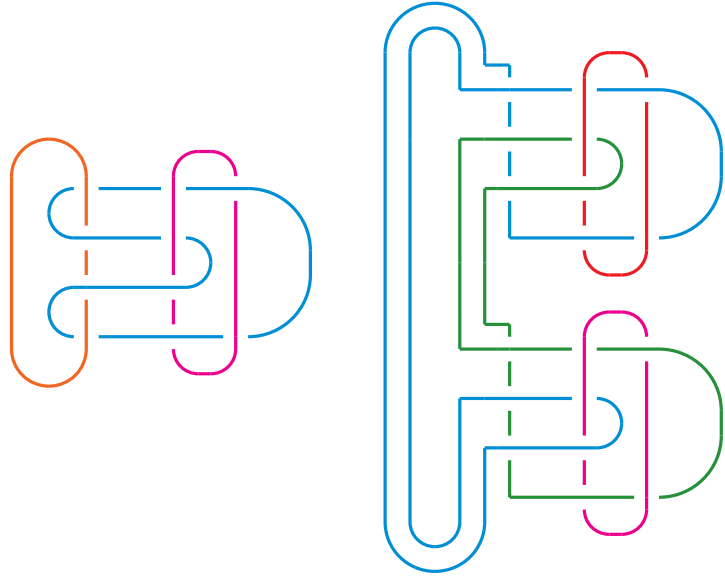}
\end{center}
\caption{The link $L8a16$ after an isotopy (left), and the 2-fold cover branched over $L_3$ (right).}
\label{figure:L8a16-cover-over-L3}
\end{figure}

\subsection{The link $L9a46$}\label{section:link-L9a46}

The link $L9a46$ is shown on the left of
Figure~\ref{figure:L9a46-one}.  We claim that $\sp(L9a46)=3$.  It is
not hard to find three changes which suffice.  For example, in
Figure~\ref{figure:L9a46-one}, change the crossings where $L_1$ passes
under~$L_2$.

Note that $|\lk(L_1,L_2)| =1$.  Therefore if one crossing change
suffices, it must be between $L_1$ and~$L_2$.  We need to show that
this is not possible.  For this purpose we apply
Corollary~\ref{corollary:covering-link-zero-lk-no} again.  We will
take a 2-fold covering link branched over~$L_3$.  In preparation for
this, the link from the left of Figure~\ref{figure:L9a46-one} is
shown, after an isotopy, on the right of
Figure~\ref{figure:L9a46-one}.

\begin{figure}[H]
  \labellist
  \small\hair 0mm
  \pinlabel {$L_1$} at 20 220
  \pinlabel {$L_2$} at 198 100
  \pinlabel {$L_3$} at 190 265
  \endlabellist
\includegraphics[scale=.65]{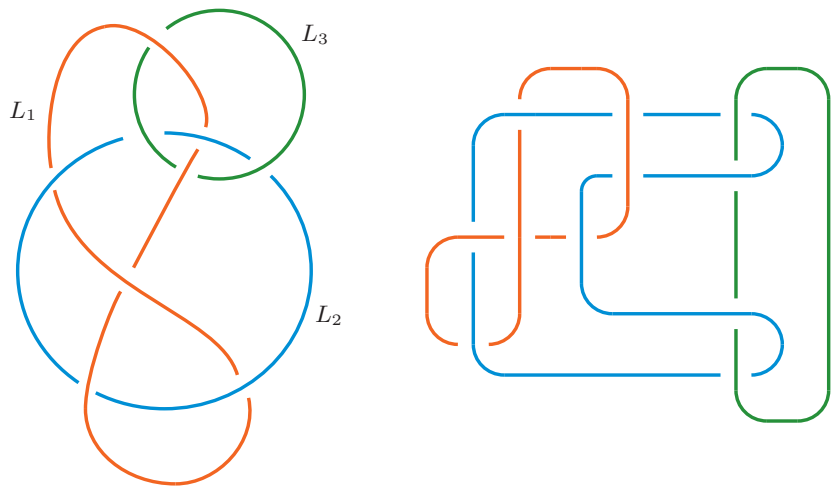}
\caption{Left: the link $L9a46$.  Right: the link $L9a46$ after an isotopy.}
\label{figure:L9a46-one}
\end{figure}

Taking the cover branched over the right hand component of the link on
the right of Figure~\ref{figure:L9a46-one}, we obtain the 2-fold
covering link $J$ shown in
Figure~\ref{figure:L9a46-covering-link-over-L3}.

\begin{figure}[H]
  \labellist
  \small\hair 0mm
  \pinlabel {$J_1$} at 35 140
  \pinlabel {$J_2$} at 9 72
  \pinlabel {$J_3$} at 126 210
  \pinlabel {$J_4$} at 396 139
  \endlabellist
\begin{center}
\includegraphics[scale=.7]{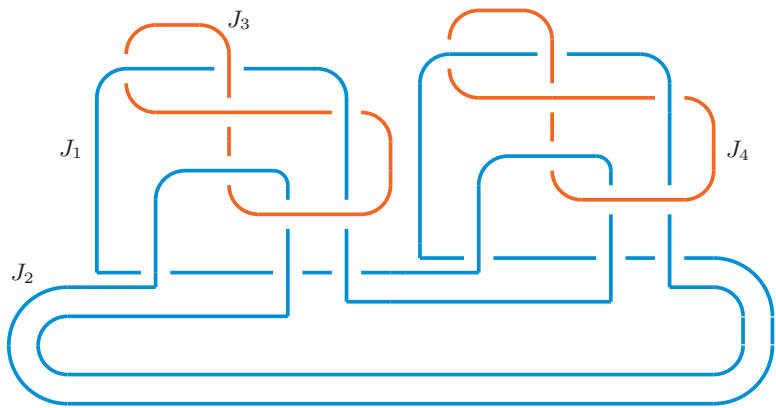}
\end{center}
\medskip

\caption{The 2-fold covering link of the link $L9a46$ from the right of Figure~\ref{figure:L9a46-one}.}
\label{figure:L9a46-covering-link-over-L3}
\end{figure}

We need to see that the link $J$ of
Figure~\ref{figure:L9a46-covering-link-over-L3} has splitting number at least 6.
Observe that $|\lk(J_1,J_4)| = |\lk(J_2,J_3)| = 1$.
Moreover, the sublinks $J_1 \sqcup J_3$ and $J_2 \sqcup J_4$ are
Whitehead links.  It now follows from
Lemma~\ref{lem:splittingnumberbasics} that $\sp(J)\geq
6$.  By Corollary~\ref{corollary:covering-link-zero-lk-no}, we obtain that
$\sp_3(L9a46)\geq 3$.  It thus follows from the above discussion that
$\sp(L9a46) = 3$.

\bibliographystyle{amsalpha}
\def\MR#1{}
\bibliography{unlinkingbib1}

\providecommand{\bysame}{\leavevmode\hbox to3em{\hrulefill}\thinspace}
\providecommand{\MR}{\relax\ifhmode\unskip\space\fi MR }
\providecommand{\MRhref}[2]{%
  \href{http://www.ams.org/mathscinet-getitem?mr=#1}{#2}
}
\providecommand{\href}[2]{#2}
\begin{thebibliography}{HKL10}

\bibitem[Ada96]{Adams96}
C.~C. Adams, \emph{Splitting versus unlinking}, J. Knot Theory Ramifications
  \textbf{5} (1996), no.~3, 295--299. \MR{1405713 (97h:57011)}

\bibitem[BS13]{batson-seed}
J.~Batson and C.~Seed, \emph{A link splitting spectral sequence in {K}hovanov
  homology}, arXiv:1303.6240, 2013.

\bibitem[CG78]{CassonGordon2}
A.~Casson and C.~McA. Gordon, \emph{On slice knots in dimension three},
  Algebraic and geometric topology ({P}roc. {S}ympos. {P}ure {M}ath.,
  {S}tanford {U}niv., {S}tanford, {C}alif., 1976), {P}art 2, Proc. Sympos. Pure
  Math., XXXII, Amer. Math. Soc., Providence, R.I., 1978, pp.~39--53.

\bibitem[CG86]{CassonGordon}
{A}. {C}asson and {C}.~{M}c{A}. {G}ordon, \emph{{C}obordism of classical
  knots}, A la {R}echerche de la {T}opologie {P}erdue, Progr. Math., vol.~62,
  Birkhauser Boston, 1986, pp.~181--199.

\bibitem[CK08]{Cha-Kim:2008-1}
J.C. Cha and T.~Kim, \emph{Covering link calculus and iterated {B}ing doubles},
  Geom. Topol. \textbf{12} (2008), no.~4, 2173--2201. \MR{MR2431018
  (2009f:57007)}

\bibitem[CLa]{knotinfo}
J.~C. Cha and C.~Livingston, \emph{{K}not{I}nfo: {T}able of knot invariants},
  \texttt{http://www.\penalty100indiana.\penalty100edu/\penalty100\char`\~knotinfo/},
  August 12, 2013.

\bibitem[CLb]{linkinfo}
\bysame, \emph{{L}ink{I}nfo: {T}able of link invariants},
  \texttt{http://www.\penalty100indiana.\penalty100edu/\penalty100\char`\~linkinfo/},
  August 12, 2013.

\bibitem[Hil02]{Hi02}
J.~A. Hillman, \emph{Algebraic invariants of links}, Series on Knots and
  Everything 32 (World Scientific Publishing Co.), 2002.

\bibitem[HKL10]{HKL}
C.~Herald, P.~Kirk, and C.~Livingston, \emph{{M}etabelian representations,
  twisted {A}lexander polynomials, knot slicing, and mutation}, Mathematische
  Zeitschrift \textbf{265} (2010), no.~4, 925--949.

\bibitem[HS97]{HS97}
J.~A. Hillman and M.~Sakuma, \emph{On the homology of finite abelian coverings
  of links}, Canad. Math. Bull. \textbf{40} (1997), 309--315.

\bibitem[Kaw96]{KawauchiBook96}
A.~Kawauchi, \emph{A survey of knot theory}, Birkh\"auser Verlag, Basel, 1996,
  Translated and revised from the 1990 Japanese original by the author.
  \MR{1417494 (97k:57011)}

\bibitem[KL99]{KirkLivingston}
P.~Kirk and C.~Livingston, \emph{Twisted {A}lexander invariants, {R}eidemeister
  torsion, and {C}asson-{G}ordon invariants}, Topology \textbf{38} (1999),
  no.~3, 635--661.

\bibitem[Koh91]{Kohn91}
P.~Kohn, \emph{Two-bridge links with unlinking number one}, Proc. Amer. Math.
  Soc. \textbf{113} (1991), no.~4, 1135--1147. \MR{1079893 (92c:57008)}

\bibitem[Koh93]{Kohn93}
\bysame, \emph{Unlinking two component links}, Osaka J. Math. \textbf{30}
  (1993), no.~4, 741--752. \MR{1250780 (95g:57012)}

\bibitem[Mur67]{Mu67}
K.~Murasugi, \emph{On a certain numerical invariant of link types}, Trans.
  Amer. Math. Soc. \textbf{117} (1967), 387--422.

\bibitem[Ras10]{Rasmussen-s-invariant}
J.~Rasmussen, \emph{Khovanov homology and the slice genus}, Invent. Math.
  \textbf{182} (2010), no.~2, 419--447. \MR{2729272 (2011k:57020)}

\bibitem[Shi12]{Shimizu12}
Ayaka Shimizu, \emph{The complete splitting number of a lassoed link}, Topology
  Appl. \textbf{159} (2012), no.~4, 959--965. \MR{2876702 (2012k:57016)}

\end{thebibliography}

\end{document}